\documentclass[10pt, a4paper]{article}

\usepackage{a4}
\usepackage{latexsym}
\usepackage{amssymb}
\usepackage{amsfonts}
\usepackage{amsthm}
\usepackage{amsmath}
\usepackage[pdftex]{graphicx}
\usepackage{cite}
\usepackage{cases}

\newcommand{\essinf}[1]{\underset{#1}{\mbox{ess inf}} \,}

\newcommand{\bq}{\begin{equation}}
\newcommand{\eq}{\end{equation}}

\newtheorem{theorem}{Theorem}
\newtheorem{corollary}[theorem]{Corollary}
\newtheorem{proposition}[theorem]{Proposition}

\newtheorem{definition}[theorem]{Definition}
\newtheorem{example}[theorem]{Example}
\newtheorem{remark}[theorem]{Remark}

\begin{document}
\title{On the role of Gittin’s index in singular stochastic control: semi-explicit solutions via the Wiener-Hopf factorisation.}
\author{Jenny Sexton\footnote{School of Mathematics, University of Manchester. Oxford Road, {\sc Manchester, M13 9PL, United Kingdom.} E-mail: jennifer.sexton@postgrad.manchester.ac.uk}}
\date{\today}
\maketitle

\begin{abstract}

This paper examines a class of singular stochastic control problems with convex objective functions. In Section 2, we use tools from convex analysis to derive necessary and sufficient first order conditions for this class of optimisation problems. The main result of this paper is Theorem \ref{Theorem connection} which uses results from optimal stopping to establish the link between singular stochastic control and Gittin's index without the need to appeal to the representation result in \cite{BEK}. In sections 3-5 we assume the singular control problem is driven by a L\'{e}vy process. Expressions for the Gittin's index are derived in terms of the Wiener-Hopf factorisation. This allows us to broaden the class of parameterised optimal stopping problems with explicit solutions examined in \cite{BF} and derive explicit solutions to the singular control problems studied in Section 2. In Section 4, we apply our results to the `monotone follower' problem which originates in \cite{BC} and \cite{K81}. In Section 5, we apply our results to an irreversible investment problem which has been studied in \cite{Bertola}, \cite{Kobila} and \cite{RiedelSu}.

\end{abstract}

\begin{tabbing}


\end{tabbing}

\section{Introduction}
\label{section intro}

This paper aims to characterise the solution to the following singular control problem
\begin{equation}
\inf_{\theta \in \mathcal{A}}E_{x}\left[ \int_{0}^{\infty}c\left(t,X_t,\theta_t\right) e^{-rt}\,dt+\int_{0}^{\infty}k_{t}e^{-rt}\,d\theta_{t}\right],
\end{equation}
where $\mathcal{A}$ contains all increasing c\`{a}dl\`{a}g processes, the cost function $c(t,x,\theta)$ is convex in $\theta$ and both the intervention cost process $(k_t)_{t\geq 0}$ and the state process $(X_t)_{t\geq 0}$ do not depend on the control $(\theta_t)_{t\geq 0}$. The assumptions used in the rest of this paper are discussed in detail later in this section. 

This problem is similar, but more general than the `monotone follower' problem studied in \cite{BC}, \cite{K81}, \cite{K2} and \cite{KS} as the cost function need not have the form $c(t,x,\theta) = c(t,x-\theta)$. In particular, this class of problems contains as special cases an irreversible investment problem (see \cite{Bertola}, \cite{CH2}, \cite{Kobila}, \cite{Oks} and \cite{Pham}). When $(X_t)_{t\geq 0}$ and $(k_t)_{t\geq 0}$ are Markov processes, the typical approach to solving problems of this type is to use variational inequalities (see \cite{FS} and the references therein for the general theory). More generally, reflected BSDEs may be used to characterise the solution on a finite time horizon (see \cite{HL1} and \cite{HL2}). More recently, this problem has been studied for more general state processes and/or subject to additional constraints in \cite{Bank}, \cite{FThesis} and \cite{RiedelSu}. These later results rely heavily on the representation result introduced in \cite{BEK}. 

In Section 2, we use results from convex analysis (in particular \cite{ET}) to derive first order conditions for this class of optimisation problems. These first order conditions are both necessary and sufficient, and are shown to imply the sufficient condition derived in \cite{Bank}. In particular, our first order conditions are a generalisation of the `Maximum principle' introduced in \cite{CH} for singular control problems driven by linear diffusions. On a finite time horizon, a `Maximum principle' has also been derived for more general singular control problems driven by diffusions in \cite{BM}. Historically, there appears to have been less interest in this approach than via the associated PDEs. Potentially, this is because characterising a control satisfying the maximum principle's conditions is often not much easier than solving the original problem. Although they did not refer to these first order conditions as a maximum principle, a control satisfying these first order conditions has been characterised in \cite{Bank} and \cite{RiedelSu} using the representation result introduced in \cite{BEK}. 

In fact, a connection between the representation and singular control problems is not entirely surprising, as the former is derived using a family optimal stopping problems and it is well-known that singular control and optimal stopping are intimately related (see for example \cite{KS}, \cite{BR}, \cite{B} and the references therein). In particular, this representation result appeals to the existence of a representing process (referred to as a `signal process' in \cite{BB}, \cite{BF}, \cite{CF}, \cite{FThesis}, \cite{RiedelSu} and \cite{Suo}) which is a continuous time version of Gittin's index. The Gittin's index has its origins in dynamic allocation problems (see for example \cite{GJ}, \cite{EKK1}, \cite{K}, \cite{KM}) where it is used to study the continuation region of a parameterised family of optimal stopping problems. To the best of the authors knowledge, this connection has not been exploited to provide additional examples of the representation result in \cite{BEK} except in \cite{CF} where in a finite time setting the Gittin's index is expressed in terms of the free-boundary of an associated stopping problem.

With this in mind, the main result of this paper is Theorem \ref{Theorem connection} which uses only results from the general theory of optimal stopping to show that the control satisfying the conditions of the maximum principle coincides with the Gittin's index of an associated optimal stopping problem. In particular, we use neither the representation introduced in \cite{BEK} nor any other similar results. This result has some similarities with the connection between optimal stopping games and bounded variation control in \cite{KW}. However, the case of a strictly increasing/decreasing control is not covered by their assumptions and furthermore, we are able to derive additional properties of the entire path of the optimal control. 

The results in Section 2 hold under quite general conditions on the state process and intervention costs similar to the `non-Markovian' framework used in \cite{Bank} and \cite{KW}. However, the results in Section 2 do require that the running costs depend only on the current value of the control and not the entire path, which differs from the typical approach taken with diffusions (compare with for example \cite{CF}). In Sections 3-5 we focus on the case that the state process is a L\'{e}vy process. In Section 3, we introduce a parameterised set of stopping problems of the type
\begin{equation*}
v(x,c) = \inf_{\tau \geq 0}E_x\left[ \int_0^{\tau}g(X_t)e^{-rt}\,dt +c\,e^{-r\tau}\right],
\end{equation*}
where $c\in \mathbb{R}$ and $g$ is a monotone function. The Gittin's Index associated with this problem is the smallest parameter such that it is optimal to stop immediately, i.e.
\begin{equation*}
c^*(x)=\inf\{c\in \mathbb{R}\,\vert\,v(x,c)=c\}.
\end{equation*}
This family of optimal stopping problems has classically been used to solve dynamic allocation problems (see for example \cite{EKK1}, \cite{K}, \cite{KM}) and more recently received independent attention (see for example \cite{BB}, \cite{BF}, \cite{EKK2}, \cite{Suo}). 

We derive explicit expression for the Gittin's index in terms of the Wiener-Hopf factorisation. This expression allows us to derive explicit examples of the representation studied in \cite{BEK}. These allow us to use the results from \cite{BF} to derive expressions for the value function of these optimal stopping problems, significantly extending the class of explicit examples provided in \cite{BF} and \cite{Suo}. It is perhaps interesting to note that the Gittin's index coincides with the so-called `EPV' operators studied in \cite{BY1} and \cite{BY2}. Hence \cite{BY1} and \cite{BY2} provide a wealth of additional problems where the Gittin's index plays a role. 

Due to the results in Section 2, there is a one-to-one correspondence between the explicit solutions to the stopping problems derived in Section 3 and the optimal control problem introduced in the remainder of this section. In Sections 4 and 5 we examine two examples in more detail. Section 4 focuses on the `monotone follower' problem which originates in the study of controlling the orbit of a satellite (see \cite{BC} and \cite{K2}). In this example, the running cost has an additive form where $c(t,x,\theta)=c(t,x-\theta)$. A monotone follower example features in \cite{Bank} where a guess and verify technique is used to examine the case where $c(x,x-\theta)=(x-\theta)^2/2$ and $(X_t)_{t\geq 0}$ is a Brownian motion. We are able to deal with any strictly convex $c$ and provide a very straightforward representation in the quadratic case.

The second application we examine is an irreversible investment problem. Here the running cost function takes a multiplicative form, i.e. $c(t,x,\theta)=p(\theta)q(t,x)$ where $p$ is a convex function and for each $t\geq 0$ the function $x \mapsto q(t,x)$ is monotone. This problem is studied in \cite{Bertola} and has been generalised in \cite{Kobila} and \cite{RiedelSu}, where the economics behind this problem is discussed at more length. We are able to provide explicit examples which extend beyond the `Cobb-Douglas' costs example discussed in \cite{RiedelSu}. In particular, we are not restricted to the case that the intervention costs are deterministic.

To summarise, the rest of this section introduces the framework that we use throughout. In Section 2 we derive first order conditions which are shown to be satisfied by a Gittin's index process. In Section 3 we derive explicit expressions for this Gittin's index and use it to solve a class of optimal stopping problems. Finally, in Sections 4-5 we apply these results to gain explicit expressions for the optimal control in two structurally different examples. 

Consider a set of filtered probability spaces $(\Omega, \mathcal{F}, \mathbb{F}, (P_x)_{x\in \mathbb{R}})$  and let $X=(X_t)_{t\geq 0}$ denote a L\'{e}vy process which has $X_0 = x\in \mathbb{R}$ under $P_x$ and is adapted to $\mathbb{F}$. The results in Section \ref{Section L_max} do not rely on the fact that $X$ is a L\'{e}vy process and apply more generally, however, in Sections \ref{Section L_stopping}-\ref{Section L_production} specific properties of L\'{e}vy processes are used. Take $r>0$ to be a deterministic interest rate.

\begin{definition}
\label{def_control}
A singular control $\theta$ is an increasing, adapted c\`{a}dl\`{a}g process such that
\begin{equation*}
\left\Vert \int_0^{\infty}e^{-rt}\,d\theta_t \right\Vert_{L^2(P_x)} < \infty
\end{equation*}
and $\theta_{0-}=0$. The set of all singular controls is denoted by $\mathcal{A}$.
\end{definition}

\noindent
Each singular control $\theta \in \mathcal{A}$ incurs an intervention cost $k=(k_t)_{t\geq 0}$ and a running cost described by a deterministic function $(t,x,\theta) \mapsto c(t,x,\theta)$. These costs satisfy the following assumptions. 

\medskip

\noindent
\textbf{Assumptions}.

\begin{enumerate}

\item[1)] The process $k=(k_{t})_{t \geq 0}$ is a positive c\`{a}dl\`{a}g process which satisfies 
\begin{equation*}
E_{x}\left[ \sup_{t\geq 0}\left\vert k_{t}\right\vert^{2}\right]<+\infty. 
\end{equation*}

\item[2)] For each $(t,x)\in \mathbb{R}_{+} \times \mathbb{R}$ the running cost function $\theta \mapsto c(t,x,\theta)$ has a continuous and strictly increasing first derivative with respect to $\theta$, denoted $c_{\theta}(t,x,\theta)$. It is assumed that this derivative is bounded uniformly in $t$ and $x$.

\item[3)] There exists $y\in \mathbb{R}_+$ such that for all $x\in \mathbb{R}$
\begin{equation*}
E_{x}\left[ \int_{0}^{\infty}c(t,X_t,y) e^{-rt}\,dt\right]<+\infty .
\end{equation*}

\item[4)] The indirect cost function $C: \mathbb{R}_{+} \times \mathbb{R} \times \mathbb{R}_{+} \rightarrow \mathbb{R}$ is defined using 
\begin{equation}
C(t,x,k)=\inf_{z \geq 0}\left( c(t,x,z)+rkz \right) \label{indirect cost}
\end{equation}
and it is assumed that 
\begin{equation*}
E_x\left[ \int_0^\infty C(t,X_t,k_t)e^{-rt}\,dt\right] > -\infty.
\end{equation*}
\end{enumerate}
 
The total expected cost resulting from choosing the singular control $\theta \in \mathcal{A}$ is determined by the functional $J:\mathcal{A}\rightarrow \mathbb{R} \cup \{+\infty\}$ defined as
\begin{equation}
J(\theta) =E_{x}\left[ \int_{0}^{\infty}c\left( t,X_t,\theta_t\right) e^{-rt}\,dt+\int_{0}^{\infty}k_{t-}e^{-rt}\,d\theta_{t}\right]. \label{objective function}
\end{equation}
Throughout the rest of this paper, we are interested in solving the following optimisation problem
\begin{equation}
\inf_{\theta \in \mathcal{A}}J\left( \theta \right) . \label{minimisation problem}
\end{equation}
The expression for $J$ in (\ref{objective function}) suggests a natural domain of controls $\theta$ to consider would be those such that $\int_0^\infty e^{-rt} d\theta_t$ is in $L^1(P_x)$ rather than in $L^2(P_x)$ as is imposed in Definition \ref{def_control}. However the tools used in this paper require this stronger condition, for instance to ensure that the G\^ateaux differential of $J$ exists (see Definition \ref{def Gat} and Proposition \ref{Proposition Gateaux}).

As in \cite{Bank}, the positivity of the intervention costs $k$ can be relaxed to bounded from below without any significant change to the subsequent results. Assumptions 1 and 3 ensure that the effective domain of $J$ is non-empty as there exists $y\in \mathbb{R}$ such that $J\left( \theta^{y}\right) <+\infty$ where $\theta^y_t = y$ for all $t\geq 0$. Assumption 2 states that for each $(t,x)\in \mathbb{R}_{+}\times \mathbb{R}$ the function $\theta \mapsto c(t,x,\theta)$ is convex and consequently the objective function (\ref{objective function}) is convex. The assumption that the derivative of $\theta \mapsto c(t,x,\theta)$ is bounded is imposed for technical convenience in Section \ref{Section L_max}. In the most prominent examples this does not hold, but it is straightforward to remove this assumption (see the proof of Theorem \ref{theorem monotone} in Section \ref{Section L_mono}).

Assumption 4 differs from assumption (iv) in \cite{Bank} in order to allow for running cost functions which are monotone, but still ensures that $J$ is bounded from below. To see this fix an arbitrary $\theta \in \mathcal{A}$ and observe that the objective function (\ref{objective function}) satisfies
\begin{align*}
J(\theta) &\geq E_x\left[ \int_0^{\infty}c(t,X_t,\theta_t)e^{-rt}\,dt + \int_0^{\infty}rk_t\theta_t e^{-rt}\,dt - \int_0^{\infty}rk_t\theta_t e^{-rt}\,dt \right] \\ &\geq E_x\left[ \int_0^{\infty}C(t,X_t,k_t)e^{-rt}\,dt \right] - E_x\left[\int_0^{\infty}rk_t\theta_t e^{-rt}\,dt \right] .
\end{align*}
The first term is greater than $-\infty$ due to Assumption 4. The second term satisfies
\begin{align*}
E_x\left[\int_0^{\infty}k_t\theta_t re^{-rt}\,dt \right] &\leq E_x\left[\sup_{t\geq 0}\vert k_t\vert\int_0^{\infty}\theta_t re^{-rt}\,dt \right] \leq E_x\left[\sup_{t\geq 0}\vert k_t\vert\int_0^{\infty}e^{-rt}\,d\theta_t \right] \\
&\leq \left\Vert \sup_{t\geq 0}\vert k_t\vert \right\Vert_{L^2(P_x)}\left\Vert \int_0^{\infty}e^{-rt}\,\theta_t \right\Vert_{L^2(P_x)}
\end{align*}
which is finite due to Assumption 1. Hence, these assumptions ensure that (\ref{minimisation problem}) is finite. 


\begin{remark}
The use of c\`{a}dl\`{a}g rather than predictable controls allows the first order conditions derived in Theorem \ref{Theorem Maximum principle} to be slightly stronger than the conditions derived in \cite{CH}. This is because by choosing a c\`{a}dl\`{a}g control, it is possible to prevent the L\'{e}vy process $X$ from leaving a certain region of the state space: when $X$ jumps out of the region the control can immediately act to `push' the process back so that the net effect is the controlled process never leaves the region. It could be argued that from the perspective of applications this is not really possible: the jump of $X$ has to observed first before action can be taken. Taking this perspective means that only predictable controls are viable. When the intervention cost process is continuous this poses no additional complication as the left continuous version $(\theta_{t-})_{t \geq 0}$ of any control is predictable and yields the same payoff.
\end{remark}

\section{Maximum principle}

\label{Section L_max}

In this section we derive first order conditions for optimality in the singular control problem (\ref{minimisation problem}) where the objective function $J$ is as defined in (\ref{objective function}). These first order conditions are similar to those for one-dimensional linear diffusions derived in \cite{CH}. Firstly it will be shown that $J$ is G\^{a}teaux differentiable. Secondly, we use a result from convex analysis to derive first order conditions for optimality involving the G\^{a}teuax differential of $J$. These conditions are linked to the first order conditions derived in \cite{Bank}.

Let $\mathcal{R}^2$ denote the vector space of all c\`adl\`ag adapted processes $Z$ equipped with the norm
\begin{equation*}
\Vert Z\Vert_{\mathcal{R}^2} = \left\Vert \sup_{t\geq 0}\vert Z_t\vert \right\Vert_{L^2(P_x)} < \infty.
\end{equation*}
Let $\mathcal{V} \subset \mathcal{R}^2$ denote the space of all processes $A$ which can be decomposed as 
$A=A^+ + A^-$ where $A^{\pm}$ are processes of integrable variation, $A^-$ is predictable with $A_0^-=0$, $A^+$ is optional and purely discontinuous with no jump at infinity and
\begin{equation*}
\int_{(0,\infty]}\vert dA^-_t \vert + \int_{[0,\infty)}\vert dA^+_t\vert \in L^2(P_x).
\end{equation*}
Definition \ref{def_control} states that a singular control $\theta\in \mathcal{A}$ is a positive increasing process which is locally in $\mathcal{V}$. Associate with each $\theta\in \mathcal{A}$ a process $A^{\theta}$ defined using $dA^{\theta}_t=e^{-rt}\,d\theta_t$ for all $t\geq 0$. Due to Definition \ref{def_control} for each $\theta \in \mathcal{A}$, $A^{\theta}$ is a positive and increasing element of $\mathcal{V}$. 

Define a bilinear mapping $B: \mathcal{R}^2 \times \mathcal{A}\rightarrow \mathbb{R}$ using
\begin{equation}
B(Z,\theta) = E_x\left[\int_0^{\infty}Z_te^{-rt}\,d\theta_t\right] . \label{bilinear map}
\end{equation}
It has been shown in \cite{DM} VII.3 Theorem 65 that all of the continuous linear functionals on $\mathcal{R}^2$ can be represented in the form
\begin{equation}
\langle Z,A \rangle =E_x\left[\int_{(0,\infty]}Z_{t-}\,dA^-_t + \int_{[0,\infty)}Z_t\,dA^+_t\right]  \label{linear_Op}
\end{equation}
where $A^{\pm}$ satisfy the assumptions outlined above. Note that $B(Z,\theta)=\langle Z,A^{\theta}\rangle$ but not all of the linear operators on $\mathcal{R}^2$ can be written in this way. This bilinear mapping can be extended to allow that $Z$ is only locally in $\mathcal{R}^2$ in the obvious manner, if it is allowed to take the values $\pm \infty.$

We use the following definition of G\^{a}teaux differentiability on vector spaces which is from \cite{Werner} III.5.1.

\begin{definition}
\label{def Gat}
Let $\mathcal{X}$ be a normed vector space and take $F:\mathcal{X}\rightarrow \mathbb{R} \cup \{\pm\infty\}$. When it exists the (possibly infinite) limit
\begin{equation*}
\lim_{\lambda \downarrow 0}\frac{F(u+\lambda v)-F(u)}{\lambda}
\end{equation*}
is referred to as the directional derivative of $F$ at $u$ in the direction $v$. The functional $F$ is G\^{a}teaux differentiable at $u\in \mathcal{X}$ when $\vert F(u)\vert <\infty$ and there exists a continuous linear functional $DF(u):\mathcal{X} \rightarrow \mathbb{R}$ such that
\begin{equation*}
\lim_{\lambda \downarrow 0}\frac{F(u+\lambda v)-F(u)}{\lambda} = DF(u)(v)
\end{equation*}
for all $v\in \mathcal{X}$ such that this limit is finite. The functional $DF(u)$ is referred to as the G\^{a}teaux differential of $F$ at $u$.
\end{definition}

The next proposition provides a condition under which the objective functional $J$ is G\^{a}teaux differentiable on the set $\mathcal{A}$ introduced in Definition \ref{def_control}. Within the class of control problems studied in this paper, this G\^{a}teaux differential coincides with the subgradient process used in \cite{Bank}. A convex combination of singular controls is taken by setting 
\begin{equation}
\theta^{\lambda} = (1-\lambda)\theta + \lambda\eta
\end{equation}
for any $\theta,\eta \in \mathcal{A}$ and $\lambda \in [0,1]$. The objective functional $J$ defined in (\ref{objective function}) is extended onto the processes which are locally in $\mathcal{R}^2$ (denoted $\mathcal{R}^2_{\mathrm{loc}}$) by setting $J(\eta)=+\infty$ for $\eta \in \mathcal{R}^2_{\mathrm{loc}}\setminus \mathcal{A}$.

\begin{proposition}
\label{Proposition Gateaux} 
Take $\theta\in \mathcal{A}$ such that $J(\theta)<\infty$ then for all $\eta\in\mathcal{A}$ the objective function $J$ has a directional derivative in the direction $\eta-\theta \in \mathcal{V}$ at $\theta \in \mathcal{A}$ given by
\begin{equation*}
\lim_{\lambda \downarrow 0}\frac{J(\theta^{\lambda })-J(\theta)}{\lambda}=E_x\left[ \int_0^{
\infty}(k_{t}+p_{t}^{\theta})e^{-rt}\,d(\eta_t-\theta_t)\right],
\end{equation*}
where for each $s \geq 0$
\begin{equation}
p_s^{\theta}=E_x\left[ \left. \int_s^{\infty}c_{\theta}(u,X_u,\theta_u)e^{-r(u-s)}\,du\,\right\vert\,\mathcal{F}_{s}\right] .  \label{adjoint convex}
\end{equation}
In particular, $J$ is G\^{a}teaux differentiable at $\theta$ with G\^{a}teaux differential $DJ(\theta)(\cdot):=B(k+p^{\theta},\cdot)$.
\end{proposition}

\begin{proof}
The convexity of $\eta \mapsto c(t,x,\eta)$ implies that
\begin{align*}
E_x\left[\int_0^{\infty}c_{\theta}(t,X_t,\theta_t)(\eta_t-\theta_t)e^{-rt}\,dt\right] &\leq \frac{1}{\lambda}E_x\left[\int_0^{\infty}(c(t,X_t,\theta_t^{\lambda})-c(t,X_t,\theta_t))e^{-rt}\,dt\right] \\ &\leq E_x\left[\int_0^{\infty}c_{\theta}(t,X_t,\theta_t^{\lambda})(\eta_t-\theta_t)e^{-rt}\,dt\right] .
\end{align*}
Using Assumption 2 and that $\theta, \eta \in \mathcal{A}$ we can apply the dominated convergence theorem to derive that
\begin{equation*}
\lim_{\lambda \downarrow 0}\frac{1}{\lambda}E_x\left[\int_0^{\infty}(c(t,X_t,\theta_t^{\lambda})-c(t,X_t,\theta_t))e^{-rt}\,dt\right] = E_x\left[\int_0^{\infty}c_{\theta}(t,X_t,\theta_t)(\eta_t-\theta_t)e^{-rt}\,dt\right] . 
\end{equation*}
Applying Fubini's Theorem yields 
\begin{align*}
E_x\left[ \int_0^{\infty}c_{\theta}(t,X_t,\theta_t)
\int_{0}^{t}d(\eta_s-\theta_s)e^{-rt}\,dt\right] &= E_x\left[
\int_{0}^{\infty}\int_{s}^{\infty}c_{\theta}(t,X_t,\theta_t)e^{-rt}\,dt\,d(\eta_s-\theta_s)\right] \\
&= E_x\left[\int_{0}^{\infty}p_{s}^{\theta}e^{-rs}\,d(\eta_s-\theta_s)\right]
\end{align*}
where for the second equality \cite{J} Theorem 1.33 has been used. Observe that it follows from the definition of $\theta^{\lambda}$ that
\begin{align*}
\lim_{\lambda \downarrow 0}\frac{J(\theta^{\lambda})-J(\theta)}{\lambda} &=  E_x\left[\int_{0}^{\infty}p_{s}^{\theta}e^{-rs}\,d(\eta_s-\theta_s)\right]+\lim_{\lambda \downarrow 0}\frac{1}{\lambda}E_x\left[\int_0^{\infty}k_{s}e^{-rs}\,d(\theta^{\lambda}_s -\theta_s) \right] \\ &=  E_x\left[\int_0^{\infty}(p_{s}^{\theta}+k_{s})e^{-rs}\,d(\eta_s-\theta_s)\right].
\end{align*}
Under Assumption 2, $p^{\theta}\in \mathcal{R}^2$ so it follows from Assumption 1 that 
\begin{equation*}
B(k+p^{\theta},\eta-\theta) \leq \left(\Vert k \Vert_{\mathcal{R}^2}+ \Vert p^{\theta} \Vert_{\mathcal{R}^2}\right)\left(\left\Vert \int_{0}^{\infty} e^{-rt} \,d\eta_t \right\Vert_{L^{2}(P_x)} + \left\Vert \int_{0}^{\infty}e^{-rt}\,d\theta_t\right\Vert_{L^{2}(P_x)}\right)
\end{equation*}
illustrating that when $\eta,\theta \in \mathcal{A}$ the directional derivative is a bounded (and hence continuous) linear functional as required. 
\end{proof}
 
\medskip

It is a consequence of the convexity of $J$ that
\begin{equation*}
J(\eta)-J(\theta)=\frac{\lambda J(\eta)+ (1-\lambda)J(\theta)-J(\theta)}{\lambda} \geq \frac{J(\theta^{\lambda})-J(\theta)}{\lambda}.
\end{equation*}
for all $\lambda \in [0,1]$. Thus, Proposition \ref{Proposition Gateaux} shows that the directional derivative of $J$ at $\theta$ in the direction $\eta-\theta$ is a `subgradient' of the objective functional in the sense that 
\begin{equation*}
J(\eta)-J(\theta) \geq E_x\left[ \int_0^{\infty}(k_{t}+p_{t}^{\theta}) e^{-rt} \,d(\eta_t-\theta_t)\right]
\end{equation*} 
for all $\eta \in \mathcal{A}$. When $J$ is G\^{a}teaux differentiable at $\theta$ this subgradient is unique according to \cite{ET} Proposition 1.5.3. With a minor abuse of terminology, the process $DJ(\theta)=(DJ(\theta)_t)_{t\geq 0}$ defined using $DJ(\theta)_t=k_t+p^{\theta}_t$ for all $t\geq 0$ will be referred to as the subgradient process associated with $J$ as the subgradient of $J$ in an arbitrary direction can be identified using this subgradient process. The next result is part of \cite{ET} Proposition 2.2.1 and states a very simple necessary and sufficient condition for optimality.

\begin{proposition}
\label{proposition ET} The following two statements are equivalent:

\begin{enumerate}
\item[(1)] $\theta^{\ast} \in \mathcal{A}$ attains the infimum in (\ref{minimisation problem});

\item[(2)] For all $\eta \in \mathcal{A}$ such that $\vert J(\eta)\vert <\infty$ we have 
\begin{equation}
E_x\left[ \int_0^{\infty}DJ(\theta^{\ast})_{t}e^{-rt} \,d(\eta_t-\theta^{\ast}_t)\right] \geq 0 .  \label{VI}
\end{equation}
\end{enumerate}

\end{proposition}

\begin{proof}
First assume that $\theta^{\ast}$ attains the infimum in (\ref{minimisation problem}). Then $J(\theta^{\ast}) \leq J((1-\lambda)\theta^{\ast}+\lambda\eta) $ for all $\eta \in \mathcal{A}$ and all $\lambda \in [0,1]$. In particular,
\begin{equation*}
\frac{J(\theta^{\ast}+\lambda(\eta-\theta^{\ast}))-J(\theta^{\ast})}{\lambda}
\geq 0
\end{equation*}
for all $\eta \in \mathcal{A}$, hence taking the limit as $\lambda \downarrow 0$ shows that it follows from Propostion \ref{Proposition Gateaux} that $(1) \Rightarrow (2)$. Next suppose that (2) holds then due to the convexity of $J$
\begin{equation*}
0\leq E_x\left[ \int_0^{\infty}DJ(\theta^{\ast})_{t}e^{-rt} \,d(\eta_t-\theta^{\ast}_t) \right] \leq \frac{\lambda J(\eta)+ (1-\lambda)J(\theta^{\ast})-J(\theta^{\ast})}{\lambda} = J(\eta)-J(\theta^{\ast}). 
\end{equation*}
and hence $(2) \Rightarrow (1)$. 
\end{proof}

\medskip

Condition (\ref{VI}) may be interpreted as stating that there is no direction in which an incremental change in $\theta^{\ast}$ can be made which results in a lower payoff and hence $\theta^{\ast}$ is a local minimum. Since $J$ is convex, this local minimum $\theta^{\ast}$ satisfying (\ref{VI}) is a global minimum. The next result is the provides first order conditions for for the minimimsation problem (\ref{minimisation problem}) and is a generalisation of \cite{BM} Theorem 7 and \cite{CH} Theorem 4.2 to our setting. 

\begin{theorem}
\label{Theorem Maximum principle} 
The control $\theta^{\ast }\in \mathcal{A}$ is optimal for the singular stochastic control problem (\ref{minimisation problem}) if and only if
\begin{equation}
DJ(\theta^{\ast})_t \geq 0 \qquad \forall\,t\geq 0 \quad P_x\text{\textrm{-}}\mathrm{a.s.}  \label{positivity}
\end{equation}
and
\begin{equation}
E_x\left[ \int_{0}^{\infty}DJ(\theta^{\ast})_{t-}e^{-rt}\,d\theta _{t}^{\ast} \right] =0 .  \label{flat off}
\end{equation}
\end{theorem}

\begin{proof}
Assume that (\ref{positivity}) and (\ref{flat off}) hold, then for an arbitrary $\eta$ such that $J(\eta)<\infty$ we have
\begin{equation*}
B(DJ(\theta^{\ast}),\eta-\theta^{\ast}) = B(DJ(\theta^{\ast}),\eta)-B(DJ(\theta^{\ast}),\theta^{\ast}) =B(DJ(\theta^{\ast}),\eta) \geq 0
\end{equation*}
which is (\ref{VI}). Hence it follows from Proposition \ref{proposition ET} that (\ref{positivity}) and (\ref{flat off}) imply that $\theta^{\ast}$ attains the infimum in (\ref{minimisation problem}). 

Now suppose that $\theta^{\ast}$ is optimal, it was shown in Proposition \ref{proposition ET} that the optimal control $\theta^{\ast}$ satisfies (\ref{VI}). Consider a singular control $\varphi$, defined via $d\varphi _{t}=\mathbb{I}_{[DJ(\theta^{\ast})_{t-}<0]}\,d\theta^{\ast}_t$. For this particular control, (\ref{VI}) implies that
\begin{equation*}
0 \leq B(DJ(\theta^{\ast}),\varphi-\theta^{\ast}) = B(-DJ(\theta^{\ast}) \mathbb{I}_{[DJ(\theta^{\ast})\geq 0]},\theta^{\ast}) .
\end{equation*}
To avoid a contradiction the right-hand side must be equal to zero, from
which we deduce that 
\begin{equation}
B(DJ(\theta^{\ast})^+,\theta^{\ast})=0. \label{pre flat off}
\end{equation} 

Consider another  singular control $\nu$ defined using $d\nu_t=d\theta^{\ast}_t+\mathbb{I}_{[DJ(\theta^{\ast})_{t-}<0]}\,dt$. For this control, (\ref{VI}) implies that
\begin{equation*}
0 \leq B(DJ(\theta^{\ast}),\nu -\theta^{\ast}) = E_x\left[\int_0^{\infty}DJ(\theta^{\ast})_{t-}
\mathbb{I}_{[DJ(\theta^{\ast})_{t}< 0]}e^{-rt}\,dt\right]
\end{equation*}
and to avoid a contradiction $DJ(\theta^{\ast})_t\geq 0$ for almost every $t\geq 0$. 

Let $\Lambda := \{ t\geq 0\,\vert\,DJ(\theta^{\ast})_{t}^{-}>0\}$ and for any $\theta \in \mathcal{A}$ suppose that $\widehat{\theta}$ satisfies $d\widehat{\theta}_t =\mathbb{I}_{\Lambda}(t)\,d\theta_{t}$. This construction ensures that $B(DJ(\theta^{\ast})^+,\widehat{\theta})=0$. Consequently, combining 
(\ref{VI}) and (\ref{pre flat off}) yields
\begin{equation}
0 \leq B(DJ(\theta^{\ast})^+ - DJ(\theta^{\ast})^-,\widehat{\theta}-\theta^{\ast}) = B(DJ(\theta^{\ast})^-,\theta^{\ast}-\widehat{\theta}) \label{step max1}.
\end{equation}
Take $\widehat{\phi} \in \mathcal{A}$ such that $P_x(\int_0^{\infty}\mathbb{I}_{\Lambda}(t)\,d(\widehat{\phi}_t - \theta^{\ast}_t)>0)=1$. For this particular $\widehat{\phi} \in \mathcal{A}$
\begin{equation}
B(DJ(\theta^{\ast})^-,\theta^{\ast}-\widehat{\phi}) = -B(DJ(\theta^{\ast})^-\mathbb{I}_{\Lambda},\theta^{\ast}-\widehat{\phi}) < 0
\end{equation}
which contradicts (\ref{step max1}) unless $DJ(\theta^{\ast})_{t}^{-}=0$ $P_x$-a.s. Hence thia proves that if $\theta^{\ast}$ is optimal then (\ref{positivity}) follows. Furthermore, this observation combined with (\ref{pre flat off}) implies that (\ref{flat off}) holds.
\end{proof}

\medskip

The conditions in the previous theorem are stronger than the conditions found in \cite{CH} because here the singular controls are c\`{a}dl\`{a}g processes rather than predictable therefore the optimal subgradient process may not jump into the negative half-plane as $\theta^{\ast}$ may immediately adjust to prevent this. 

\begin{remark}
\label{remark intuition}
The G\^{a}teaux differential $DJ(\theta)_t=k_t + p_t^{\theta}$ can be interpreted as the marginal cost of an increase in the control $\theta$ at time $t\geq 0$. A small increase in $\theta_t$ will incur intervention cost $k_t$ and have a marginal impact on all of the future running costs, namely $p^{\theta}_t$. When $DJ(\theta)_t > 0$ the intervention costs exceed the marginal benefits of increasing $\theta_t$ which suggests that a smaller total cost can be achieved by leaving $\theta_t$ at its current level. Similarly when $DJ(\theta)_t \leq 0$ the intervention costs are less than the marginal benefits of increasing $\theta_t$ which suggests that smaller total cost can be achieved by making a small increase in $\theta_t$. Condition (\ref{positivity}) tells us that the optimal control ensures that we never enter the region where the marginal benefit of increasing $\theta^*$ exceeds the intervention costs. Moreover, condition (\ref{flat off}) tells us that $\theta^*$ is the smallest control which satisfies (\ref{positivity}).
\end{remark}

The next corollary shows that Theorem \ref{Theorem Maximum principle} contains as a special case \cite{Bank} Theorem 2.2 with no upper limit on the amount of control that can be used.

\begin{corollary}
\label{corollary Snell}
The optimal control $\theta^{\ast}\in\mathcal{A}$ for the singular stochastic control problem (\ref{minimisation problem}) is such that $\theta^{\ast}$ is flat-off the set
\begin{equation*}
\left\{t\geq 0\,\left\vert\,DJ(\theta^{\ast})_{t}=S(\theta^{\ast})_{t}\right. \right\}
\end{equation*}
where the process $S(\theta^{\ast})$ is the `Snell envelope' of the optimal subgradient process $DJ(\theta^{\ast})$, i.e.
\begin{equation}
S(\theta^{\ast})_{\tau}=\essinf{\sigma \geq \tau} E_x\left[\left. DJ(\theta^{\ast})_{\sigma} e^{-r(\sigma-\tau)} \,\right\vert\, \mathcal{F}_{\tau}\right] \label{snell}
\end{equation}
for all stopping times $\tau$.
\end{corollary}

\begin{proof}
It follows from the definition (\ref{adjoint convex}) and Assumption 1 that 
\begin{equation*}
\lim_{t\rightarrow\infty}e^{-rt}DJ(\theta^{\ast})_t = \lim_{t\rightarrow\infty}e^{-rt}p^{\theta^{\ast}}_t  + \lim_{t\rightarrow\infty}k_t e^{-rt} =0 .
\end{equation*}
Hence, $S(\theta^{\ast})_{\tau} \leq E_x\left[ \left. \lim_{t \rightarrow +\infty}DJ(\theta^{\ast})_{t}e^{-r(t-\tau)}\,\right\vert\, \mathcal{F}_{\tau}\right]=0$ for any finite stopping time $\tau$ so the Snell envelope takes values in $(-\infty ,0]$. As $\theta^{\ast}$ is optimal it follows from (\ref{positivity}) that $DJ(\theta^{\ast})_{\sigma}\geq 0$ for all $\sigma \geq \tau$. Consequently, $S(\theta^{\ast})_{\tau}=0$ for all stopping times $\tau$ so the result follows from (\ref{flat off}). 
\end{proof}

\medskip

Actually we may go one step further and link the first order conditions of Theorem \ref{Theorem Maximum principle} to a different family of optimal stopping problems. This result is then an extension of the connection between optimal stopping and singular stochastic control identified in \cite{KS} for the `monotone follower problem' driven by Brownian motion. A similar connection between bounded variation control and optimal stopping games is identified in \cite{KW}. 

\begin{theorem}
\label{Theorem connection}
Define a family of optimal stopping problems indexed by $l\in \mathbb{R}$ via
\begin{equation}
u(t;l)e^{-rt} := \essinf{\sigma \geq t}E_x\left[\left.k_{\sigma}e^{-r\sigma}-\int_t^{\sigma}c_{\theta}(u,X_u,l)e^{-ru}\,du\,\right\vert\,\mathcal{F}_t \right] \label{family of probs}. 
\end{equation}
Suppose that $\theta^*\in \mathcal{A}$ is optimal for the singular control problem (\ref{minimisation problem}) then
\begin{equation}
S(\theta^*)_t = u(t,\theta^*_t) + p_t^{\theta^*} \quad \forall t\geq 0. \label{coincide}
\end{equation}
For each $t\geq 0$, let 
\begin{equation}
l_t = \sup\{ l\in\mathbb{R}\,\vert\,u(t;l)=k_t\} \label{parameter}
\end{equation}
where we understand $\sup \emptyset=-\infty$. Let $\widehat{\theta}_t = \sup_{0 \leq u \leq t} l_u \vee 0$ for all $t \geq 0$ and suppose that $\widehat{\theta}\in \mathcal{A}$. Then, the singular control $\widehat{\theta}$ attains the infimum in (\ref{minimisation problem}).
\end{theorem}

\begin{proof}
The proof is quite lengthy so is broken down into a number of steps,  however, we first need to introduce some notation including an additional family of auxiliary optimal stopping problems. Let $\sigma^{\ast}(t,l):=\inf\{s\geq t\,\vert\,u(s;l)=k_s\}$. The general theory of optimal stopping (see Theorem D.9 \cite{KS_Fin}) states that this stopping time attains the essential infimum in (\ref{family of probs}). For $l\geq 0$ define another family of optimal stopping problems using
\begin{equation}
S(\theta^* \vee l)_t e^{-rt} := \essinf{\sigma\geq t}E_x\left[\left. DJ(\theta^*\vee l)_{\sigma}e^{-r\sigma}\,\right\vert\,\mathcal{F}_t\right] \label{S_l family}
\end{equation}
for each $t\geq 0$. Since $\theta^*_t \vee l\geq \theta_t^*$ for all $t\geq 0$, $c_{\theta}(t,X_t,\theta^*_t\vee l)\geq c_{\theta}(t,X_t,\theta^*_t)$ for all $t\geq 0$ which implies that $p_t^{\theta^*\vee l}\geq p_t^{\theta^*}$. Consequently, 
\begin{equation*}
DJ(\theta^*\vee l)_t := k_t + p_t^{\theta^*\vee l}\geq k_t + p_t^{\theta^*} = DJ(\theta^*)_t \geq 0
\end{equation*}
where the final inequality follows from (\ref{positivity}) due to the assumed optimality of $\theta^*$. Hence, the argument in the proof of Corollary \ref{corollary Snell} implies that $S(\theta^*\vee l)_t =0$ for all $t\geq 0$. Hence, the essential infimum in (\ref{S_l family}) is attained by the stopping times $\sigma(t,l):=\inf\{s\geq t\,\vert\,DJ(\theta^{\ast}\vee l)_s =0\}$.

Step 1: We aim to show that for each $(t,l)\in \mathbb{R}_+ \times \mathbb{R}_+$, $\sigma^*(t,\theta_t^*\vee l)\leq \sigma(t,l)$. For arbitrary $(t,l)\in \mathbb{R}_+ \times \mathbb{R}$, it follows from the definition (\ref{S_l family}) that
\begin{align*}
S(\theta^* \vee l)_t e^{-rt} &= \essinf{\sigma\geq t}E_x\left[\left. k_{\sigma}e^{-r\sigma}+\int_{\sigma}^{\infty}c_{\theta}(u,X_u,\theta_u^*\vee l)e^{-ru}\,du\,\right\vert\,\mathcal{F}_t\right] \\
&= \essinf{\sigma\geq t}E_x\left[\left. k_{\sigma}e^{-r\sigma}-\int_t^{\sigma}c_{\theta}(u,X_u,\theta_u^*\vee l)e^{-ru}\,du\,\right\vert\,\mathcal{F}_t\right] + p_t^{\theta^*\vee l}e^{-rt} \\
&\leq \essinf{\sigma\geq t}E_x\left[\left. k_{\sigma}e^{-r\sigma}-\int_t^{\sigma}c_{\theta}(u,X_u,\theta^*_t\vee l)e^{-ru}\,du\,\right\vert\,\mathcal{F}_t\right] + p_t^{\theta^*\vee l}e^{-rt} \\ &= (u(t,\theta^*_t\vee l)+p_t^{\theta^*\vee l})e^{-rt}.
\end{align*}
where the second equality follows from the definition of $p^{\theta^*\vee l}$, the inequality uses that $\theta^*_u\vee l\geq \theta_t^* \vee l$ as well as the assumption that $l\mapsto c_{\theta}(\cdot,\cdot,l)$ is strictly increasing. Hence we can conclude that
\begin{equation}
0=S(\theta^* \vee l)_t \leq u(t,\theta_t^*\vee l)+p_t^{\theta^*\vee l} \label{step_1}
\end{equation}
for each $t\geq 0$. Take $k \geq l$, then it follows from Assumption 2 that 
\begin{align*}
u(t;l)e^{-rt} &= E_x\left[\left.k_{\sigma^{\ast}(t,l)}e^{-r\sigma^{\ast}(t,l)}-\int_t^{\sigma^{\ast}(t,l)}c_{\theta}(u,X_u,l)e^{-ru}\,du\,\right\vert\,\mathcal{F}_t \right] \\ &\geq E_x\left[\left.k_{\sigma^{\ast}(t,l)}e^{-r\sigma^{\ast}(t,l)}-\int_t^{\sigma^{\ast}(t,l)}c_{\theta}(u,X_u,k)e^{-ru}\,du\,\right\vert\,\mathcal{F}_t \right] \geq u(t;k)e^{-rt}.
\end{align*}
for each $t\geq 0$. We conclude that for each $t\geq 0$, the mapping $l\mapsto u(t,l)$ is monotonically decreasing. The stopping times $\sigma(t,l)$ attain (\ref{S_l family}) so (\ref{step_1}) implies
\begin{align}
0=DJ(\theta^* \vee l)_{\sigma(t,l)}&=k_{\sigma(t,l)} +p_{\sigma(t,l)}^{\theta^*\vee l} \notag \\ &\leq u(\sigma(t,l),\theta^*_{\sigma(t,l)}\vee l)+p_{\sigma(t,c)}^{\theta^*\vee l} \notag \\
&\leq u(\sigma(t,l),\theta^*_t\vee l)+p_{\sigma(t,c)}^{\theta^*\vee l} \label{inter_1}
\end{align}
where the final inequality follows since $l\mapsto u(\sigma(t,l),l)$ is decreasing. Moreover, the stopping times $\sigma(t,l)$ are suboptimal for the family of stopping problems (\ref{family of probs}) hence $u(\sigma(t,l),\theta^*_t\vee l) \leq k_{\sigma(t,l)}$. Hence, it follows from (\ref{inter_1}) that $u(\sigma(t,l),\theta^*_t\vee l) = k_{\sigma(t,l)}$ which implies that $\sigma^*(t,\theta_t^*\vee l)\leq \sigma(t,l)$.

Step 2: Let $\sigma^+(l):= \inf\{t\geq 0\,\vert\,\theta^*_t>l\}$ denote the strict level passage time of the optimal control. In this step we show that $\sigma(0,l)\leq \sigma^+(l)$ for all $l\geq 0$. It follows from the first order conditions derived in Theorem \ref{Theorem Maximum principle} that for any stopping time $\tau$
\begin{equation}
P_x\left(\int_{\tau}^{\infty}DJ(\theta^*)_t
e^{-rt}\,d\theta_t^* =0 \right)=1. \label{inter_flat}
\end{equation}
Let $\sigma(l):=\inf\{t\geq 0\,\vert\,\theta_t^*\geq l\}$ so that for all $t\geq \sigma(l)$, $\theta_t^*\vee l =\theta_t^*$ and $DJ(\theta^*\vee l)_t = DJ(\theta^*)_t$. Moreover, 
\begin{align*}
\int_{\sigma(l)}^{\infty}DJ(\theta^*)_t e^{-rt}\,d\theta_t^* &= \int_{\sigma(l)}^{\infty} DJ(\theta^*\vee l)_t e^{-rt}\,d\theta_t^* \\
&=\int_0^{\infty}DJ(\theta^*\vee l)_t\mathbb{I}_{[t\geq \sigma(l)]} e^{-rt}\,d\theta_t^* \\
&=\int_0^{\infty}DJ(\theta^*\vee l)_t\mathbb{I}_{[\theta^*_t\geq l]} e^{-rt}\,d\theta_t^* \\ &=\int_0^{\infty}DJ(\theta^*\vee l)_t e^{-rt}\,d(\theta_t^*\vee l)
\end{align*}
and hence it follows from (\ref{inter_flat}) that
\begin{equation}
P_x\left(\int_0^{\infty}DJ(\theta^*\vee l)_t e^{-rt}\,d(\theta_t^*\vee l) =0\right)=1. \label{inter_2}
\end{equation}
Since for each $l\geq 0$, $DJ(\theta^*\vee l)_t \geq 0$ for all $t\geq 0$, this condition tells us that the process $(\theta^*_t\vee l)_{t\geq 0}$ only increases on the set $\Gamma(l)=\{t\geq 0\,\vert\,DJ(\theta^*\vee l)_t =0\}$.
As the first point of increase in $\theta^*\vee l$ need not occur immediately upon entering the set $\Gamma(l)$ we have shown that
\begin{equation*}
\sigma(0,l):= \inf\{t\geq 0\,\vert\,DJ(\theta^*\vee l)_t=0 \}\leq \sigma^+(l)
\end{equation*}
for all $l\geq 0$. 

Step 3: In this step we show that for any $(t,l)\in \mathbb{R}_+ \times \mathbb{R}_+$, $\sigma(t,\theta_t^*\vee l)\leq \sigma^*(t,(\theta^*_t\vee l)+\varepsilon)$ for all $\varepsilon >0$. It follows from the definition (\ref{family of probs}) that
\begin{align*}
u(t,(\theta_t^*\vee l)+\varepsilon)e^{-rt} &:= \essinf{\sigma\geq t}E_x\left[\left. k_{\sigma}e^{-r\sigma}-\int_t^{\sigma}c_{\theta}(u,X_u,(\theta_t^*\vee l)+\varepsilon)e^{-ru}\,du\,\right\vert\,\mathcal{F}_t\right]\\
&\leq E_x\left[\left. k_{\sigma(t,l)}e^{-r\sigma(t,l)}-\int_t^{\sigma(t,l)}c_{\theta}(u,X_u,(\theta_t^*\vee l)+\varepsilon)e^{-ru}\,du \,\right\vert\,\mathcal{F}_t\right].
\end{align*}
Moreover, 
\begin{align*}
p_t^{\theta^*\vee l}&:= E_x\left[\left.\int_t^{\infty}c_{\theta}(u,X_u,\theta_u^*\vee l)e^{-ru}\,du\,\right\vert\,\mathcal{F}_t \right] \\
&= E_x\left[\left.\int_{\sigma(t,l)}^{\infty}c_{\theta}(u,X_u,\theta_u^*\vee l)e^{-ru}\,du + \int_t^{\sigma(t,l)}c_{\theta}(u,X_u,\theta_t^*\vee l)e^{-ru}\,du\,\right\vert\,\mathcal{F}_t \right]
\end{align*}
because (\ref{inter_2}) implies that $\theta_u^* \vee l=\theta_t^*\vee l$ for all $u\in [t,\sigma(t,l))$. Since $\sigma(t,l)$ attains (\ref{S_l family}) it follows that
\begin{multline*}
(u(t,(\theta_t^*\vee l)+\varepsilon)+p_t^{\theta^*\vee l})e^{-rt} \leq S(\theta^*\vee l)_te^{-rt} \\ - E_x\left[\left.\int_t^{\sigma(t,l)}(c_{\theta}(u,X_u,(\theta^*_t\vee l)+\varepsilon)-c_{\theta}(u,X_u,\theta_t^*\vee l))e^{-ru}\,du \,\right\vert\,\mathcal{F}_t\right].
\end{multline*}
The final term is positive due to Assumption 2. Hence we may conclude that for any $\varepsilon>0$,
\begin{equation}
u(t,(\theta_t^*\vee l)+\varepsilon)+p_t^{\theta^*\vee l} \leq S(\theta^*\vee l)_t \label{step_2}
\end{equation}
for any $(t,l)\in \mathbb{R}_+\times \mathbb{R}_+$. Fix any $\varepsilon >0$ and recall that the stopping time $\sigma^*(t,(\theta^*_t\vee l)+\varepsilon)$ attains $u(t,(\theta^*_t\vee l)+\varepsilon)$. Consequently it follows from (\ref{step_2}) that
\begin{align*}
DJ(\theta^*\vee l)_{\sigma^*(t,(\theta^*_t\vee l)+\varepsilon)} &= k_{\sigma^*(t,(\theta^*_t\vee l)+\varepsilon)} + p_{\sigma^*(t,(\theta^*_t\vee l)+\varepsilon)}^{\theta^*\vee l} \\ &= u(\sigma^*(t,(\theta^*_t\vee l)+\varepsilon),(\theta^*_t\vee l)+\varepsilon) + p_{\sigma^*(t,(\theta^*_t\vee l)+\varepsilon)}^{\theta^*\vee l} \\& \leq S(\theta^*\vee l)_{\sigma^*(t,(\theta^*_t\vee l)+\varepsilon)}
\end{align*}
which shows that $\sigma(t,\theta_t^*\vee l)\leq \sigma^*(t,(\theta^*_t\vee l)+\varepsilon)$ for any $\varepsilon >0$.

Step 4: In this step we show that $\sigma^*(t,\theta_t^*\vee l)=\sigma(t,l)$ for each $(t,l)\in \mathbb{R}_+\times\mathbb{R}_+$. It is sufficient to show that
\begin{equation}
0=S(\theta^*\vee l)_t = u(t,\theta^*_t\vee l) + p_t^{\theta^*\vee l} \label{step_3}
\end{equation}
for each $(t,l)\in \mathbb{R}_+\times\mathbb{R}_+$. Taking $l=0$ in (\ref{step_3}) we obtain (\ref{coincide}). Combining (\ref{step_1}) and (\ref{step_2}) yields
\begin{equation}
u(t,(\theta_t^*\vee l)+\varepsilon) + p_t^{\theta^*\vee l} \leq S(\theta^*\vee l)_t \leq u(t,\theta^*_t\vee l) + p_t^{\theta^*\vee l}. \label{estimate}
\end{equation}
Taking the limit as $\varepsilon \downarrow 0$ on the left hand side of (\ref{estimate}) gives
$u(t,(\theta_t^*\vee l)+) + p_t^{\theta^*\vee l} \leq S(\theta^*\vee l)_t$ from which we can conclude (\ref{step_3}) holds if the mapping $l\mapsto u(t,l)$ is continuous.

Fix $(t,l)\in \mathbb{R}_+\times\mathbb{R}_+$, $\varepsilon \neq 0$ and for ease of notation, let $\sigma^*:=\sigma^*(t,l+\varepsilon)$. The stopping time $\sigma^*$ is suboptimal for $u(t,l)$ and attains the essential infimun in $u(t,l+\varepsilon)$ so
\begin{align*}
u(t,l+\varepsilon)-u(t,l) &\geq E_x\left[\left.k_{\sigma^*}e^{-r\sigma^*}-\int_t^{\sigma^*}c_{\theta}(u,X_u,l+\varepsilon)e^{-ru}\,du\,\right\vert\,\mathcal{F}_t\right] \\ &\qquad\qquad- E_x\left[\left.k_{\sigma^*}e^{-r\sigma^*}-\int_t^{\sigma^*}c_{\theta}(u,X_u,l)e^{-ru}\,du\,\right\vert\,\mathcal{F}_t\right] \\ &= E_x\left[\left.\int_t^{\sigma^*}(c_{\theta}(u,X_u,l)-c_{\theta}(u,X_u,l+\varepsilon))e^{-ru}\,du\,\right\vert\,\mathcal{F}_t\right] \\ &\geq E_x\left[\left.\int_t^{\infty}(c_{\theta}(u,X_u,l)-c_{\theta}(u,X_u,l+\varepsilon))e^{-ru}\,du\,\right\vert\,\mathcal{F}_t\right] .
\end{align*}
where the final inequality uses that we have assumed that $c_{\theta}$ is increasing in $l$. As it has been assumed that $c_{\theta}(t,x,l)$ is uniformly bounded in $(t,x)$ it follows from the dominated convergence theorem that
\begin{align*}
\lim_{\varepsilon\rightarrow 0}u(t,l+\varepsilon)-u(t,l) &\geq  \lim_{\varepsilon\rightarrow 0}E_x\left[\left.\int_t^{
\infty}(c_{\theta}(u,X_u,l)-c_{\theta}(u,X_u,l+\varepsilon))e^{-ru}\,du\,\right\vert\,\mathcal{F}_t\right] \\
&=E_x\left[\left.\int_t^{\infty}\lim_{\varepsilon\rightarrow 0}(c_{\theta}(u,X_u,l)-c_{\theta}(u,X_u,l+\varepsilon))e^{-ru}\,du\,\right\vert\,\mathcal{F}_t\right] =0.
\end{align*}
where the final equality uses that it has been assumed that $l\mapsto c_{\theta}(t,x,l)$ is continuous. A symmetric argument can be used to show that $\lim_{\varepsilon\rightarrow 0}(u(t,l+\varepsilon)-u(t,l))\leq 0$ from which we can conclude that for each $t\geq 0$ the mapping $l\mapsto u(t,l)$ is continous as required. 

Step 5: In this step we show that the strict level passage times of the optimal control satisfy $\sigma^+(l)=\inf\{t\geq 0\,\vert\,\theta^*_t>l\}\leq \sigma^*(0,l+)$.
 
Suppose that $\theta^*_0>l$ so that $\sigma^+(l)=0$, in this case it has been shown in Step 1 and Step 2 that
\begin{equation*}
0\leq \sigma^*(0,l)\leq\sigma^*(0,\theta^*_0)\leq \sigma(0,l)\leq \sigma^+(l)=0
\end{equation*}
so $\sigma^+(l)=\sigma^*(0,l)$ as required. It remains to be shown that $\sigma^+(l)=\sigma^*(0,l)$ when $\theta_0^*\leq l$. Let $\sigma(l):= \inf\{t\geq 0\,\vert\,\theta_t^*\geq l\}$ denote the level passage times of $\theta^*$. Fix $k\geq l$, then for all $t<\sigma(k)$ 
\begin{align*}
DJ(\theta^*\vee k)_t &:= k_t + p_t^{\theta^*\vee k} =
k_t + E_x\left[\left.\int_t^{\infty}c_{\theta}(u,X_u,\theta^*_u\vee k)e^{-ru}\,du\,\right\vert\,\mathcal{F}_t\right] \\ &= k_t + E_x\left[\left.\int_t^{\sigma(k)}c_{\theta}(u,X_u,k)e^{-ru}\,du + \int_{\sigma(k)}^{\infty}c_{\theta}(u,X_u,\theta^*_u)e^{-ru}\,du\,\right\vert\,\mathcal{F}_t\right] \\ &> k_t + E_x\left[\left.\int_t^{\infty}c_{\theta}(u,X_u,\theta^*_u)e^{-ru}\,du\,\right\vert\,\mathcal{F}_t\right] =DJ(\theta^*)_t \geq 0
\end{align*}
where the assumption that $c_{\theta}$ is strictly increasing has been used. It follows that $\sigma(k)\leq \sigma(0,k)$ because by definition $\sigma(0,k)$ attains the essential infimum in the stopping problem $S(\theta^*\vee k)_0$. Hence, for each $l\geq 0$ such that  $\theta_0^*\leq l$ the strict level passage time $\sigma^+(l)$ satisfies 
\begin{equation}
\sigma^+(l)=\lim_{k\downarrow l}\sigma(k)\leq \lim_{k \downarrow l}\sigma(0,k) = \lim_{k \downarrow l}\sigma^*(0,k)=\sigma^*(0,l+). \label{inter_3}
\end{equation}
The second to last equality in (\ref{inter_3}) uses that it was shown in Step 4 that $\sigma^*(0,\theta_0^*\vee l)=\sigma(0,l)$.

Step 6: Define a candidate optimal control $\widehat{\theta}$ using $\widehat{\theta}_t = \sup_{0\leq u\leq t}l_u\vee 0$. We assume that $\widehat{\theta}\in \mathcal{A}$ so that $\widehat{\theta}$ satisfies the integrability condition introduced in Definition \ref{def_control} and is right-continuous. The assumption that the candidate control is right-continuous is not very restrictive since controls are increasing processes so $\widehat{\theta}$ and its right-continuous modification only differ on a countable set of points. Furthermore this modificaton does not alter strict level passage times of the candidate control.

In this step we aim to show that $P_x(\theta^*_t =\widehat{\theta}_t \quad \forall t\geq 0)=1$. Recall from Steps 2 and 5 that for all $l \in \mathbb{R}$
\begin{equation}
\sigma^*(0,l) \leq \sigma^+(l) \leq \lim_{k \downarrow l} \sigma^*(0,k).\label{inter_4}
\end{equation}
Consider a path where $\theta^*_t =\widehat{\theta}_t$ for all $t \geq 0$ does not hold, then there exists a $t'\geq 0$ such that either $\theta^*_{t'}<\widehat{\theta}_{t'}$ or $\theta^*_{t'}>\widehat{\theta}_{t'}$. For the first case, take some $c \in (\theta^*_{t'},\widehat{\theta}_{t'}) \cap \mathbb{Q}_{>0}$. By right-continuity of $\theta^*$ it follows from $\theta^*_{t'}<c$ that $\sigma^+(c)>t_0$. Furthermore, for $k \in (c,\widehat{\theta}_{t'})$ we have 
\[ t' \geq \inf \{ t \geq 0 \, | \, \widehat{\theta}_t > k \} = \inf \{ t \geq 0 \, | \, l_t \vee 0 > k \} = \inf \{ t \geq 0 \, | \, l_t > k \} \geq \sigma^*(0,k) \]
and hence $\lim_{k \downarrow c} \sigma^*(0,k) \leq t' < \sigma^+(c)$. For the second case, take some $c \in (\widehat{\theta}_{t'},\theta^*_{t'}) \cap \mathbb{Q}_{>0}$. Since $\theta^*_{t'}>c$ we have $\sigma^+(c) \leq t'$. Furthemore, by right-continuity of $\widehat{\theta}$ there exists $\delta>0$ such that $\widehat{\theta}_{t'+\delta}<c$. This implies

\[ \sup_{0 \leq u \leq t'+\delta} l_u \leq \sup_{0 \leq u \leq t'+\delta} (l_u \vee 0) <c \]
and hence $\sigma^*(0,c) \geq t'+\delta > t' \geq \sigma^+(c)$.
The above shows that
\begin{equation*}
\{\theta^*_t =\widehat{\theta}_t \quad \forall t\geq 0\}^c \subseteq \bigcup_{q \in \mathbb{Q}_{> 0}} \left( \{ \lim_{k \downarrow q} \sigma^*(0,k)<\sigma^+(q) \} \cup \{ \sigma^*(0,q)>\sigma^+(q) \} \right)
\end{equation*}
hence
\begin{align*}
1-P_x(\theta^*_t =\widehat{\theta}_t \quad \forall t\geq 0) \leq \sum_{q \in \mathbb{Q}_{> 0}} P_x \left( \lim_{k \downarrow q} \sigma^*(0,k)<\sigma^+(q) \right) + P_x \left( \sigma^*(0,q)>\sigma^+(q) \right) =0
\end{align*}
where the final equality uses (\ref{inter_4}).
\end{proof}

\medskip 

The singular control attaining (\ref{minimisation problem}) was shown in \cite{Bank} to be related to the representation studied in \cite{BEK}. However, the connection to the optimal stopping problems (\ref{family of probs}) is not derived in \cite{Bank}, despite the role of these problems in proving the representation result. The assumptions imposed in this section on the function $c_{\theta}$ differ from those used in \cite{BEK} so we are unable to prove Theorem \ref{Theorem connection} using the approach taken in \cite{Bank}.

The previous theorem links the solution to the singular control problem (\ref{minimisation problem}) to the stopping regions of the optimal stopping problems (\ref{family of probs}). Corollary \ref{corollary Snell} states that the points of increase in $\theta^*$ are included in the stopping region of the optimal stopping problem (\ref{snell}). As mentioned in Remark \ref{remark intuition}, the stopping problems (\ref{snell}) involve minimising the marginal impact of intervention. However, the marginal change in the future running costs $p^{\theta^*}$ includes the entire future path of $\theta^*$ so the optimal stopping time attaining (\ref{snell}) both determines the current value of $\theta^*$ as well as being determined by all future values of $\theta^*$. The objective function of the stopping problems (\ref{family of probs})
\begin{equation*}
E_x\left[\left.k_{\sigma}e^{-r\sigma} - \int_t^{\sigma}c_{\theta}(u,X_u,\theta_t^*)e^{-ru}\,du\,\right\vert\,\mathcal{F}_t\right]
\end{equation*}
is the marginal cost of postponing making the next small change in $\theta^*$ until the time $\sigma \geq t$. Hence, the problem $u(t,\theta^*_t)$ is attained by the stopping time which minimises the marginal cost of waiting prior to the next point of increase in $\theta^*$ and depends only on the current amount of control that has been used.

\section{Associated optimal stopping problems}

\label{Section L_stopping}

In this section we focus on providing explicit solutions to a family of optimal stopping problems indexed by $c\in \mathbb{R}$ of the form
\begin{equation}
v(x,c)=\inf_{\tau \geq 0}E_{x}\left[\int_0^{\tau}g(X_t)e^{-rt}\,dt + ce^{-r\tau}\right]
\label{family stopping problems}
\end{equation}
where $g$ is a continuous function and $X$ is a L\'{e}vy process. The solution to these problems can be used to provide explicit solutions to the singular control problem (\ref{minimisation problem}) in some special cases as illustrated in Sections \ref{Section L_mono} and \ref{Section L_production}. These problems are also of independent interest as they have been used to solve dynamic allocation problems, see \cite{EKK1}, \cite{GJ}, \cite{K}, \cite{KM} and the references therein. Define a function $G(x):=E_x[\int_0^{\infty}g(X_t)e^{-rt}\,dt]$ and assume that $\vert G(x)\vert <+\infty$. Applying the strong Markov property these problems can be rewritten as 
\begin{equation}
G(x)-v(x,c)=\sup_{\tau \geq 0}E_{x}\left[(G(X_{\tau})-c)e^{-r\tau}\right] . \label{stopping again}
\end{equation}
Parameterised families of optimal stopping problems of this type have been studied in \cite{BB}, \cite{BF}, \cite{EKK2}, \cite{Suo}.

The running supremum and infimum of the L\'{e}vy process $X$ are denoted
\begin{equation*}
\overline{X}_t = \sup_{0\leq s\leq t}X_s \quad , \quad \underline{X}_t = \inf_{0\leq s\leq t}X_s .
\end{equation*}
Let $X'=(X'_t)_{t\geq 0}$ to be an independent copy of $X$ and let $T(r),T'(r) \sim \exp(r)$ be an pair of independent exponentially distributed random variables which are independent of both $X$ and $X'$. Let $\mathbb{F}=(\mathcal{F}_t)_{t\geq 0}$ denote the natural filtration of $X$ and define a pair of functions 
\begin{equation}
\kappa(x):=\frac{1}{r} E_0\left[g\left(x+\underline{X}'_{T'(r)}\right)\right] \quad , \quad \mu(x):=\frac{1}{r}E_0\left[g\left(x+ \overline{X}'_{T'(r)}\right)\right] . \label{def kappa}
\end{equation}
Although, whether $X$ or $X'$ is used here initially appears of little consequence in what follows we will often take $x=X_t$ for some $t\geq 0$ in which case the independence of $X$ and $X'$ is relevant. Under the additional assumption that $g(x)\geq 0$ for all $x\in \mathbb{R}$ the form of function $\kappa$ is used in \cite{KM} to solve a dynamic allocation problem. The functions $\kappa$ and $\mu$ are closely related to the so-called `EPV' operators used in \cite{BY1}, \cite{BY2}. The first result in this section provides a representation of the function $G(x)$ which is used in the sequel. 

\begin{proposition}
\label{prop kappa representation of G}
Suppose that $G(x)=E_x[\int_0^{\infty}g(X_t)e^{-rt}\,dt]$ for a function $g$ which is continuous and monotone. When $g$ is increasing the function $G$ can be written as
\begin{equation}
G(x)=E_x\left[\int_0^{\infty}\sup_{0\leq s \leq t}\kappa(X_s) re^{-rt}\,dt \right] \label{representation G}
\end{equation}
where $\kappa(x)$ is as defined in (\ref{def kappa}). Similarly, when $g$ is decreasing the function $G$ can be written as
\begin{equation}
G(x)=E_x\left[\int_0^{\infty}\sup_{0\leq s \leq t}\mu(X_s) re^{-rt}\,dt \right] 
\label{representation G mu}
\end{equation}
where $\mu(x)$ is as defined in (\ref{def kappa}).
\end{proposition}

\begin{proof}
Assume that $g$ is increasing, the case that $g$ is decreasing follows using a symmetric argument. Let $T(r) \sim \exp(r)$ be another exponentially distributed random variable which is independent of $T'(r)$, $X$ and $X'$, then
\begin{equation*}
G(x)=E_x\left[\int_0^{\infty}g(X_t)e^{-rt}\,dt\right] = \frac{1}{r}E_0\left[g(x+X_{T(r)})\right] = \frac{1}{r}E_0\left[g(x+\underline{X}'_{T'(r)}+\overline{X}_{T(r)})\right] ,
\end{equation*}
where for the final step we have used that $X-\underline{X}\sim \overline{X}$ (see \cite{Kyp} Lemma 3.5). Furthermore, using the tower property
\begin{align*}
\frac{1}{r} E_0\left[g(x+\underline{X}'_{T'(r)} +\overline{X}_{T(r)})\right] &= E_0\left[\int_0^{\infty} g(x+\overline{X}_t+\underline{X}'_{T'(r)})e^{-rt}\,dt\right] \\
&=
E_0\left[\int_0^{\infty}E_0\left[\left.g(x+\overline{X}_t+\underline{X}'_{T'(r)})\,\right\vert\,\mathcal{F}_t\right]e^{-rt}\,dt\right] .
\end{align*}
The definition of $\kappa$ provided in (\ref{def kappa}) and the independence of $X_t$, $X'_{T'(r)}$ implies that
\begin{equation*}
G(x)= E_0\left[\int_0^{\infty}\kappa(x+\overline{X}_t)re^{-rt}\,dt\right] .
\end{equation*}
Take $x\leq y$, then by assumption $g(x+ \underline{X}'_{T'(r)}) \leq g(y+ \underline{X}'_{T'(r)})$ $P_0$-a.s. so it follows that the function $\kappa$ is increasing. Hence we may conclude that
\begin{equation*}
G(x)= E_0\left[\int_0^{\infty}\sup_{0\leq s \leq t}\kappa(x+X_s)re^{-rt}\,dt\right] .
\end{equation*}
\end{proof}

\medskip

The previous proposition provides a class of functions for which an explicit form of the represention result in \cite{BEK} can be provided. Theorem 1 in \cite{BEK} states that $\kappa$, $\mu$ are the unique functions such that $G$ can be represented in the form (\ref{representation G})/(\ref{representation G mu}). It has been shown in \cite{BF} that the representation result in \cite{BEK} can be used to provide a `level crossing principle' for this family of optimal stopping problems.

\begin{theorem}[\cite{BF} Theorem 1.3]
\label{Theorem BankFollmer}
Suppose that $Y$ is an optional process such that 
\begin{equation*}
P(\lim_{t\rightarrow \infty}Y_t =0)=1
\end{equation*}
which can be represented as
\begin{equation}
Y_{\tau} = E\left[\left.\int_{(\tau,+\infty]}\sup_{\tau \leq v < t}\xi_v re^{-rt}\,dt \,\right\vert\,\mathcal{F}_{\tau} \right]
\end{equation}
for all stopping times $\tau$, where $\xi$ is a progressively measurable process with upper right continuous paths such that this conditional expectation is well defined for all $\tau$.
Then the level passage time
\begin{equation*}
\tau^{\ast} = \inf\{t\geq 0\,\vert\,\xi_t\geq 0\}
\end{equation*}
is optimal for the problem $\sup_{\tau\geq 0}E[Y_{\tau}]$.
\end{theorem}

In \cite{BF} and \cite{Suo} families of perpetual American put and call options written on exponential L\'{e}vy processes have been explicitly solved by guessing and verifying the form of the process $\xi$ in Theorem \ref{Theorem BankFollmer}. The following corollary provides explicit solutions to the family of optimal stopping problems (\ref{family stopping problems}) under the assumptions of Proposition \ref{prop kappa representation of G}. 

\begin{corollary}
\label{corollary stopping}
Suppose that $G(x)=E_x[\int_0^{\infty}g(X_t)e^{-rt}\,dt]$ where $g$ is increasing and satisfies the assumptions in Proposition \ref{prop kappa representation of G}. Consider the optimal stopping problem (\ref{family stopping problems}). The stopping and continuation regions for the stopping problem with a given $c\in \mathbb{R}$ are
\begin{equation}
\mathcal{D}_c := \left\{ x\in \mathbb{R}\,\vert\, \kappa(x)\geq c\right\} \quad , \quad
\mathcal{C}_c := \left\{ x\in \mathbb{R}\,\vert\, \kappa(x)< c\right\} . \label{stopping region}
\end{equation}
and the optimal stopping time for this problem is $\tau^{\ast}(c)=\inf\{ t\geq 0\,\vert\,\kappa(X_t)\geq c\}$.
\end{corollary}

\begin{proof}
It was shown in (\ref{stopping again}) that the family of optimal stopping problems (\ref{family stopping problems}) can be expressed as
\begin{equation*}
G(x)-v(x,c)=\sup_{\tau \geq 0}E_{x}\left[(G(X_{\tau})-c)e^{-r\tau}\right] .
\end{equation*} 
Let $F(x)=G(x)-c$ so that $F(x)=E_x[\int_0^{\infty}f(X_t)e^{-rt}\,dt]$ where $f(x)=g(x)-rc$. Hence it follows from Proposition \ref{prop kappa representation of G} and the strong Markov property that 
\begin{equation*}
(G(X_t)-c)e^{-rt} = E_x\left[\left.\int_t^{\infty}\sup_{t\leq s \leq u}\left(\kappa(X_s)-c\right)re^{-ru}\,du\,\right\vert\,\mathcal{F}_t \right] .
\end{equation*}
The function $\kappa$ inherits the assumed continuity of $g$ so $(\kappa(X_t))_{t\geq 0}$ is a right-continuous process and hence Theorem \ref{Theorem BankFollmer} implies that the optimal stopping time attaining $v(x,c)$ is 
\begin{equation*}
\tau^{\ast}(c):=\inf\{ t\geq 0\,\vert\,\kappa(X_t)-c\geq 0\}.
\end{equation*} 
The stopping and continuation regions of the stopping problems (\ref{family stopping problems}) are defined as
\begin{equation*}
\mathcal{D}_c := \left\{ x\in \mathbb{R}\,\vert\, v(x,c)=c\right\} \quad , \quad
\mathcal{C}_c := \left\{ x\in \mathbb{R}\,\vert\, v(x,c)< c\right\}
\end{equation*}
so it immediately follows that these can be expressed as in (\ref{stopping region}).
\end{proof}

\medskip

When $g$ is increasing, the function $\kappa$ is sometimes referred to as the `Gittin's index' associated with the family of optimal stopping problems (\ref{family stopping problems}) in the sense that it can be represented as
\begin{equation}
\kappa(x)= \sup\{c\in \mathbb{R}\,\vert\,v(x,c)=c\} . \label{signal}
\end{equation}
The function $\kappa$ can be calculated explicitly when the Wiener-Hopf factorisation of the  L\'{e}vy process $X$ is known, for example when $X$ is spectrally one-sided or meromorphic (see \cite{Kyp}, \cite{KKP} for a discussion of properties of such processes). Explicit forms of the function $\kappa$ have been derived when $X$ is continuous or spectrally one-sided in \cite{KM}.

\begin{corollary}
\label{corollary value}
Under the assumptions of Corollary \ref{corollary stopping}, let $x^{\ast}(c)=\inf\{x\in \mathbb{R}\,\vert\,\kappa(x) \geq c\}$ and $T_{y}:=\inf\{ t\geq 0\,\vert\, X_t \geq y\}$. 
The value function $v(x,c)$ of the stopping problems (\ref{family stopping problems}) can be expressed as
\begin{align*}
G(x)-v(x,c)&= \frac{1}{r}E_x\left[\left(g(X_{T(r)}-cr)\right)\mathbb{I}_{[\overline{X}_{T(r)}\geq x^{\ast}(c)]} \right] \\&= E_x\left[ \int_0^{\infty}\sup_{0\leq u \leq t}(\kappa(X_u) -rc)^+ e^{-rt}\,dt \right].
\end{align*}
\end{corollary}

\begin{proof}
It has been shown in Corollary \ref{corollary stopping} that
\begin{equation*}
G(x)-v(x,c)=E_x\left[\left(G(X_{\tau^{\ast}(c)})-c\right)e^{-r\tau^{\ast}(c)}\right]
\end{equation*}
where $\tau^{\ast}(c):=\inf\{ t\geq 0\,\vert\, \kappa(X_t)\geq c\}=\inf\{ t\geq 0\,\vert\, X_t\geq x^{\ast}(c)\}=:T_{x^{\ast}(c)}$. The assumed form of $G$ and the strong Markov property imply that
\begin{align*}
E_x\left[\left(G(X_{\tau^{\ast}(c)})-c\right)e^{-r\tau^{\ast}(c)}\right] &=
E_x\left[ E_x\left[\left.\int_{T_{x^{\ast}(c)}}^{\infty}(g(X_u) -rc)e^{-ru}\,du\,\right\vert\, \mathcal{F}_{T_{x^{\ast}(c)}}\right]\right] \\ &=
E_x\left[\int_0^{\infty}(g(X_u) -rc)\mathbb{I}_{[u\geq T_{x^{\ast}(c)}]}e^{-ru}\,du\right] \\ &= \frac{1}{r}E_x\left[\left(g(X_{T(r)}) -rc\right)\mathbb{I}_{[\overline{X}_{T(r)}\geq x^{\ast}(c)]}\right] .
\end{align*}
For the second form of $v(x,c)$ note that applying the strong Markov property and the representation in Proposition \ref{prop kappa representation of G} yields
\begin{align*}
E_x\left[\left(G(X_{\tau^{\ast}(c)})-c\right)e^{-r\tau^{\ast}(c)}\right] &= 
E_x\left[ E_x\left[\left.\int_{T_{x^{\ast}(c)}}^{\infty}\sup_{T_{x^{\ast}(c)}\leq u \leq t}(\kappa(X_u) -c)re^{-rt}\,dt\,\right\vert\, \mathcal{F}_{T_{x^{\ast}(c)}}\right]\right] \\ &= E_x\left[ \int_{T_{x^{\ast}(c)}}^{\infty}\sup_{T_{x^{\ast}(c)}\leq u \leq t}(\kappa(X_u) -c)re^{-rt}\,dt \right] . 
\end{align*}
Furthermore, $T_{x^{\ast}(c)}$ is a point of increase of $\kappa$ and $\kappa(X_t)-c <0$ for all $t\in [0,T_{x^{\ast}(c)})$ so it follows that
\begin{equation*}
E_x\left[ \int_{T_{x^{\ast}(c)}}^{\infty}\sup_{T_{x^{\ast}(x)}\leq u \leq t}(\kappa(X_u) -c)re^{-rt}\,dt \right] = E_x\left[ \int_0^{\infty}\sup_{0\leq u \leq t}(\kappa(X_u) -c)^+ re^{-rt}\,dt \right] .
\end{equation*}
\end{proof}

\medskip

The next example shows how the approach taken in Corollaries \ref{corollary stopping} and \ref{corollary value} can be used to value a perpetual American put option. Denote by $\psi$ the Laplace exponent of $X$ defined by
\begin{equation}
\psi(c)=\frac{1}{t}\log E_0[e^{cX_t}]. \label{def psi}
\end{equation}
The Laplace exponent need not exist for all $c\in \mathbb{R}$. To ensure that $\psi(c)$ for $c\in \mathbb{R}$ is well-defined it is sufficient to assume that
\begin{equation}
\int_{\vert x\vert > 1}e^{cx}\,\nu(dx) < \infty \label{condition EM}
\end{equation} 
where $\nu$ is the L\'{e}vy measure associated with $X$ (see \cite{Kyp} Theorem 3.6).

\begin{example}[Perpetual put option]
A perpetual American put option with strike $K$ has payoff $(K-e^{X_{\tau}})^+e^{-r\tau}$ if the holder chooses to exercise the option at the stopping time $\tau$. Consider the problem
\begin{equation}
v(x)=\sup_{\tau\geq 0}E_x\left[(K-e^{X_{\tau}})^+e^{-r\tau}\right] \label{levy put}.
\end{equation}
When $P_x$ is the statistical measure this problem can be used to derive the optimal exercise time for this option, whereas when $P_x$ is a risk-neutral measure this problem can be used to derive a (not necessarily unique) fair price for this option.
It is assumed that for $c=1$, $X$ satisfies (\ref{condition EM}) and that $r-\psi(1) >0$. This assumption ensures that $\lim_{t\rightarrow \infty}E_x[e^{-rt+ X_t}]=\lim_{t\rightarrow \infty}e^{x-(r-\psi(1))t} =0$ which rules out degenerate cases where the optimal stopping time is infinite. Applying the It\^{o} formula gives
\begin{equation*}
e^{-rt+x+X_{t}}=e^{x}-\int_{0}^{t}e^{-rs+x+X_s}(r-\psi(1))\,ds+M_{t}
\end{equation*}
where $(M_t)_{t\geq 0}$ is a $P_x$-martingale, so for all $t\geq 0$
\begin{equation*}
e^{x}=E_x\left[e^{X_{t}-rt}\right]+(r-\psi(1))E_x\left[\int_{0}^{t}e^{-rs+X_s}\,ds\right].
\end{equation*}
The monotone convergence theorem implies that 
\begin{equation*}
\lim_{t\rightarrow \infty }E_x\left[ \int_{0}^{t}e^{-rs+X_s}\,ds\right] =E_x\left[ \int_{0}^{\infty}e^{-rs+X_s}\,ds\right]
\end{equation*}
and consequently,
\begin{equation*}
e^x=E_x\left[ \int_0^{\infty}(r-\psi(1))e^{-rs+X_s}\,ds\right] .
\end{equation*}
Let 
\begin{equation}
w(x)=\sup_{\sigma\geq 0}E_x\left[(K-e^{X_{\sigma}})e^{-r\sigma}\right] \label{example forward}
\end{equation} 
and define
\begin{equation}
\alpha(x):=\frac{r-\psi(1)}{r}E_0\left[e^{x+\overline{X}'_{T'(r)}}\right] =
e^x\left(E_0[e^{\underline{X}'_{T'(r)}}]\right)^{-1} .
\end{equation}
The second equality follows from the Wiener-Hopf factorisation (see: \cite{Kyp} Theorem 6.16). The function $x\mapsto K-\alpha(x)$ is decreasing so an argument similar to that in Corollary \ref{corollary stopping} implies that the stopping time 
\begin{equation*}
\sigma^{\ast}=\inf\{t\geq 0\,\vert\,\alpha(X_t)\geq K\}=\inf\left\{t\geq 0\,\vert\,e^{X_t}\geq KE_0[e^{\underline{X}'_{T'(r)}}]\right\}
\end{equation*}
is optimal for (\ref{example forward}) which coincides with the optimal stopping time for (\ref{levy put}) derived in \cite{AlKy}, so (\ref{levy put}) and (\ref{example forward}) coincide. To justify this observation note that by assumption 
\begin{equation*}
w(x)\geq \lim_{t\rightarrow \infty}E_x\left[(K-e^{X_t})e^{-rt}\right] = 0
\end{equation*}
so $(K-e^x)^+\leq w(x)\leq v(x)$. However, it can be checked using the It\^{o} formula that the process $(w(X_t)e^{-rt})_{t\geq 0}$ is a supermartingale so Lemma 9.1 in \cite{Kyp} implies that $v=w$.
\end{example}


\section{Monotone follower problems}

\label{Section L_mono}

This section focuses on the `monotone follower problem'
\begin{equation}
V(x)=\inf_{\theta \in\mathcal{A}}J(\theta) = \inf_{\theta \in\mathcal{A}}
E_x\left[\int_{0}^{\infty}c(t,X_t-\theta_t)
e^{-rt}\,dt+\int_{0}^{\infty}k_{t}e^{-rt}\,d\theta_{t}\right] \label{monotone follower}
\end{equation}
where $y \mapsto c(t,y)$ is a continuously differentiable function which is bounded from below and has an increasing derivative. Furthermore $X$ is a L\'evy process such that
\begin{equation*}
E_x \left[ \int_0^\infty c(t,X_t) e^{-rt} dt \right] < \infty 
\end{equation*}
so that Assumptions 3 and 4 are satisfied. The first result in this section identifies the control attaining (\ref{monotone follower}) using the connection to optimal stopping discussed in Theorem \ref{Theorem connection} and the representation provided in Proposition \ref{prop kappa representation of G}.

\begin{theorem}
\label{theorem monotone}
Suppose that $\int_{x>1}x^2\,\nu(dx)$ where $\nu$ is the L\'{e}vy measure associated with $X$. Suppose $k_t=K\geq 0$, $c(t,x)=c(x)$ for all $t\geq 0$ and consider the monotone follower problem (\ref{monotone follower}). Let
\begin{equation}
\kappa(x):=\frac{1}{r}E_0\left[ c_x\left( x+\underline{X}_{T(r)}\right) \right]-K \label{kappa monotone}
\end{equation}
and suppose that there exists $x^{\ast}\in \mathbb{R}$ which solves $\kappa(x)=0$. Then the optimal control for the monotone follower problem (\ref{monotone follower}), $\theta^{\ast}$, is given by
\begin{equation*}
\theta_{t}^{\ast}=\sup_{0\leq u\leq t}(X_u-x^{\ast})^{+} \qquad \forall t\geq 0.
\end{equation*}
\end{theorem}

\begin{proof}
First we show that $\theta^* \in \mathcal{A}$ i.e. that
\begin{equation}
E_x \left[ \left( \int_0^\infty e^{-rt} d\theta^*_t \right)^2 \right]<\infty. \label{control in A} 
\end{equation}
For any $a<b$ we have 
\begin{equation*}
\theta^*_b-\theta^*_a \leq (x+\overline{X}_b)-(x+\overline{X}_a)=\overline{X}_b-\overline{X}_a \leq \sup_{t \in [a,b]} (X_t-X_a) .
\end{equation*}
For any $n \geq 1$, let $Y_k=\sup_{t \in [(k-1)/n,k/n]} (X_t-X_{(k-1)/n})$ then it follows from the definition of the Stieltjes integral that 
\begin{equation*}
E_x \left[ \left( \int_0^\infty e^{-rt} d\theta^*_t \right)^2 \right] \leq E_0 \left[ \left( \int_0^\infty e^{-rt} d\overline{X}_t \right)^2 \right] \leq E_0 \left[ \left( \sum_{k \geq 1} e^{-r(k-1)/n} Y_k \right)^2 \right].
\end{equation*}
As $X$ has stationary and independent increments the sequence $(Y_k)_{k \geq 1}$ is i.i.d. So (\ref{control in A}) follows provided $E_0[Y_1^2]=E_0[\overline{X}_{1/n}^2]<\infty$. It follows from $\int_{x>1} x^2 \nu(dx)<\infty$ that $E_0[\overline{X}_{T(q)}^2]<\infty$ for any $q>0$ (see Exercise 7.1 in \cite{Kyp}). Moreover,
\begin{equation*}
E_0[\overline{X}_{T(q)}^2] = E_0 \left[ \int_0^\infty \overline{X}_t^2 q e^{-qt} dt \right] =  \int_0^\infty E_0 \left[ \overline{X}_t^2 \right] q e^{-qt} dt 
\end{equation*}
so it follows that $E_0[Y_1^2]=E_0[\overline{X}_{1/n}^2]<\infty$.

Now we assume that $y\mapsto c(y)$ has a bounded derivative so that Assumption 2 is satisfied and we can make use of the results in Section \ref{Section L_max}. Theorem \ref{Theorem connection} suggests the optimal control problem (\ref{monotone follower}) is connected to the following family of optimal stopping problems
\begin{equation}
Y(x;l) :=\inf_{\sigma}E_x\left[ Ke^{-r\sigma}+\int_{0}^{\sigma}c_x(X_t-l)e^{-rt}\,dt\right] \label{stopping probs mono}
\end{equation}
where $l\in \mathbb{R}$. It was shown in Corollary \ref{corollary stopping} that the stopping and continuation regions for (\ref{stopping probs mono}) are
\begin{equation*}
\mathcal{D}_l := \left\{ x\in \mathbb{R}\,\vert\, \kappa(x-l)\geq 0\right\} \quad , \quad
\mathcal{C}_l := \left\{ x\in \mathbb{R}\,\vert\, \kappa(x-l)< 0\right\} .
\end{equation*}
where $\kappa$ is as defined in (\ref{kappa monotone}) (and not (\ref{def kappa})). Hence, $Y(x,l)=K$ for all $l\in \mathbb{R}$ such that $\kappa (x-l)\geq 0$. It now follows from the definition of $x^{\ast}$ that 
\begin{equation*}
x-x^{\ast} =\sup\left\{l\in \mathbb{R}\,\vert \,Y(x;l)= K\right\} .
\end{equation*}
Let $(l_t)_{t\geq 0}$ be defined as $l_t := X_t -x^{\ast}$, then it follows from Theorem \ref{Theorem connection} that the optimal control for (\ref{monotone follower}) takes the form specified.

Now suppose that the derivative of $c$ is increasing and unbounded. For any $n \geq 1$, let $c^{(n)}$ be a continuously differentiable convex function with increasing, bounded derivative such that $c^{(n)}(y)=c(y)$ for $y\in[-n,n]$ while for $y\in(-\infty,-n)$ (resp. $y\in(n,\infty)$) we have $c^{(n)}_x(y) \geq c_x(y)$ and $c^{(n)}(y) \leq c(y)$ (resp. $c^{(n)}_x(y) \leq c_x(y)$ and $c^{(n)}(y) \leq c(y)$). This sequence of functions is chosen such that $c^{(n)}_x(y) \downarrow c_x(y)$ for $y\leq 0$ and $c^{(n)}_x(y) \uparrow c_x(y)$ for $y>0$ as $n\rightarrow \infty$.

Note that $\kappa(x) \in [-\infty,\infty)$ and $\kappa$ is increasing and continuous when finite. Hence if $x^* \in \mathbb{R}$ exists such that $\kappa(x^*)=0$ then $\kappa(x)>0$ for all $x>x^*$. For each $n\geq 1$, define
\[ \kappa^{(n)}(x) = \frac{1}{r} E_0[c^{(n)}_x(x+\underline{X}_{T(r)})]-K. \]
Each of these functions $\kappa^{(n)}$ is finite, increasing and continuous. As $c^{(n)}_x(y) \geq c_x(y)$ for $y\in(-\infty,n]$ we have $\kappa^{(n)}(x) \geq \kappa(x)$ for $x \leq n$. 

In fact, for $x<n$ we have $-c^{(n+1)}_x(x+\underline{X}_{T(r)}) \geq -c^{(n)}_x(x+\underline{X}_{T(r)})$, $-c^{(n)}_x(x+\underline{X}_{T(r)}) \rightarrow -c_x(x+\underline{X}_{T(r)})$ and $-c^{(n)}_x(x+\underline{X}_{T(r)}) \geq -c_x(x)$. Hence it follows from the monotone convergence theorem that $\kappa^{(n)}(x) \downarrow \kappa(x)$. In particular, the mean value theorem implies that for $n>x^{\ast}$, $\kappa^{(n)}$ has a unique zero, denoted $x^{(n)}$. Furthermore, the sequence of functions $c^{(n)}$ has been chosen such that $x^{(n)} \uparrow x^*$ as $n \rightarrow \infty$.

Define 
\[ J^{(n)}(\theta) = E_x \left[ \int_0^\infty c^{(n)}(X_t-\theta_t) e^{-rt} dt + K \int_0^\infty e^{-rt} d\theta_t \right] \]
and
\[ \theta^{(n)}_t = \sup_{0 \leq u \leq t} (X_u-x^{(n)})^+. \]
It follows from the first part of the proof that $\inf_{\theta \in \mathcal{A}}J^{(n)}(\theta)=J^{(n)}(\theta^{(n)})$. Furthermore,
\[ \theta^{(n)}_t-\theta^{*}_t = \sup_{0 \leq u \leq t} (X_u-x^{(n)})^+ - \sup_{0 \leq u \leq t} (X_u-x^*)^+ = \left\{
\begin{array}{cl}
0 & \text{if $\overline{X}_t \leq x^{(n)}$} \\
\overline{X}_t-x^{(n)} & \text{if $\overline{X}_t \in (x^{(n)},x^*)$} \\
x^*-x^{(n)} & \text{if $\overline{X}_t \geq x^*$}
\end{array} \right. \]
so that $\theta^{(n)}_t \to \theta^{*}_t$ a.s. for any $t \geq 0$ and

\[ \int_0^\infty e^{-rt} d(\theta^{(n)}_t-\theta^{*}_t) \leq \int_0^\infty d(\theta^{(n)}_t-\theta^{*}_t) = x^*-x^{(n)} \to 0 \]
as $n \to \infty$. For any $\theta \in \mathcal{A}$, Fatou's lemma implies that we may write 
\begin{eqnarray*}
J(\theta^*)  & = & E_x \left[ \int_0^\infty c(X_t-\theta^{*}_t) e^{-rt} dt + K \int_0^\infty e^{-rt} d\theta^{*}_t \right] \notag\\
 & = & E_x \left[ \int_0^\infty \liminf_{n \to \infty} c^{(n)}(X_t-\theta^{(n)}_t) e^{-rt} dt + K \liminf_{n \to \infty} \int_0^\infty e^{-rt} d\theta^{(n)}_t \right] \notag\\
 & \leq & \liminf_{n \to \infty} E_x \left[ \int_0^\infty c^{(n)}(X_t-\theta^{(n)}_t) e^{-rt} dt + K \int_0^\infty e^{-rt} d \theta^{(n)}_t \right] \notag\\
 & = & \liminf_{n \to \infty} J^{(n)}(\theta^{(n)}) \leq \liminf_{n \to \infty} J^{(n)}(\theta) \leq J(\theta) \notag 
\end{eqnarray*}
where the final inequality uses that $J^{(n)}\leq J$ since $c^{(n)} \leq c$. Thus we may conclude that $J(\theta^*)=\inf_{\theta \in \mathcal{A}} J(\theta)$ as required. 
\end{proof}

\medskip

If the function $x \mapsto \kappa(x)$ defined in the previous result is strictly negative for all $x \in \mathbb{R}$ then it is optimal to not exercise control in the `monotone
follower problem' (\ref{monotone follower}) i.e. 
\begin{equation*}
V(x)=J(0)=E_x\left[\int_{0}^{\infty}c(t,X_t)e^{-rt}\,dt\right] .
\end{equation*}

The next example illustrates that for L\'{e}vy processes where the distribution of $\underline{X}_{T(r)}$ is explicitly known it is possible to be more precise about the control which attains the infimum in (\ref{monotone follower}).

\begin{example}[Spectrally one-sided]
\label{example spectrally one-sided}
Suppose that $X$ is spectrally positive, that is the L\'{e}vy measure $\nu$ associated with $X$ does not give weight to the negative half plane, i.e. $\nu((-\infty,0))=0$ and the paths of $X$ are not monotone. As $X$ is spectrally positive the Laplace exponent 
\begin{equation*}
\psi(\beta) = \frac{1}{t}\log E_0\left[e^{-\beta X_t}\right] 
\end{equation*}
is well-defined for all $\beta >0$. Denote the right inverse of $\psi$ by $\Phi(\alpha)=\sup\{\beta\,\vert\,\psi(\beta)=\alpha\}$ then $-\underline{X}_{T(r)}$ is exponentially distributed with parameter $\Phi(r)$, i.e.
$P_0\left( -\underline{X}_{T(r)}\in dx\right)=\Phi(r)e^{-\Phi(r)x}\,dx$ (see \cite{Kyp} Section 8.1). Thus, when $X$ is spectrally positive, the function (\ref{kappa monotone}) appearing in Theorem \ref{theorem monotone} is
\begin{equation}
\kappa(x) =\frac{\Phi(r)}{r}\int_{0}^{\infty}c_x(x-y)e^{-\Phi(r)y}\,dy - K. \label{kappa example}
\end{equation}

In particular consider the monotone follower problem with quadratic running costs 
\begin{equation*}
V(x) =\inf_{\theta \in \mathcal{A}}E\left[ \int_{0}^{\infty}\frac{1}{2}(X_t-\theta _t)^{2}\,re^{-rt}\,dt+K\int_{0}^{\infty}\,e^{-rt}\,d\theta _{t}\right]
\end{equation*}
Applying integration by parts to (\ref{kappa example}) yields
\begin{equation*}
\kappa(x)=x+\frac{1}{\Phi(r)}-K
\end{equation*}
and it follows from Theorem \ref{theorem monotone} that the optimal control for this problem is
\begin{equation*}
\theta_{t}^{\ast}=\sup_{0\leq u\leq t}(X_{u}-x^{\ast })^{+}
\end{equation*}
where $x^{\ast}=K-1/\Phi(r)$. Moreover, in the case that $X$ is a standard Brownian motion and $K=0$ we have $x^{\ast}=1/\sqrt{2r} =E_0[\underline{X}_{T(r)}]$ which coincides with the solution found in \cite{Bank}.

The process $Y$ is referred to as spectrally negative if $-Y$ is spectrally positive. In this case, let 
\begin{equation*}
\psi'(\beta) = \frac{1}{t}\log E_0\left[e^{\beta Y_t}\right] 
\end{equation*}
which is well-defined for all $\beta>0$. Denote the right inverse of $\psi'$ by $\Phi'(\alpha)=\sup\{\beta\,\vert\,\psi'(\beta)=\alpha\}$. Let $W^{(r)}:[0,+\infty) \rightarrow [0,+\infty)$ denote the scale function associated with $Y$ which is characterised as the function such that
\begin{equation*}
\int_{0}^{\infty }e^{-\beta x}W^{(r)}(x)\,dx=\frac{1}{\psi'(\beta) -r}.
\end{equation*} 
for all $\beta >\Phi'(r)$. It can be shown (see \cite{Kyp} pp. 213) that
\begin{equation}
P_0(-\underline{Y}_{T(r)}\in dx) =\frac{r}{\Phi'(r)}dW^{(r)}(x)-rW^{(r)}(x)\,dx
\end{equation}
Hence in the case that the L\'{e}vy process $X$ in the monotone follower problem (\ref{monotone follower}) is spectrally negative, the function (\ref{kappa monotone}) appearing in Theorem \ref{theorem monotone} is
\begin{equation*}
\kappa(x) =\frac{1}{\Phi'(r)}\int_{0}^{\infty}c_x(x-y)\,dW^{(r)}(y)-\int_{0}^{\infty}c_x( x-y) W^{(r)}(y)\,dy - K.
\end{equation*}
There are several examples of L\'{e}vy processes where semi-explicit expressions for the scale function are available, see for example \cite{HK}.
\end{example}

\section{Irreversible investment problems}

\label{Section L_production}

This section focuses on the problem
\begin{equation}
V(x)=\sup_{\theta \in\mathcal{A}}K(\theta) = \sup_{\theta \in\mathcal{A}}
E_x\left[\int_{0}^{\infty}p\left(\theta_t\right) q(t,X_t)e^{-rt}\,dt
-\int_{0}^{\infty}k_{t}e^{-rt}\,d\theta_{t}\right] \label{irreverible investment problem}
\end{equation}
where $p: \mathbb{R}_+ \to \mathbb{R}_+$ is a continuously differentiable, concave, increasing function such that its derivative tends to $0$ as $\theta$ tends to $\infty$. Furthermore $q: \mathbb{R}_+ \times \mathbb{R} \to \mathbb{R}_{++}$ is a continuous function so that for each $t \geq 0$ the function $x \mapsto q(t,x)$ is increasing. This problem coincides with (\ref{objective function}) when the running costs are taken to have the separable form $c(t,x,\theta)=-p(\theta)q(t,x)$. Assume that $(k_t)_{t\geq 0}$ satisfies Assumption 1 and $X$ is a L\'{e}vy process such that the function $c(t,x,\theta)=-p(\theta)q(t,x)$ satisfies Assumptions 2-4. Under these assumptions the functional $J(\theta)=-K(\theta)$ is convex and Proposition \ref{Proposition Gateaux} and Theorem \ref{Theorem Maximum principle} can be applied.

This type of problem can be used to model the optimal expansion of production capacity. The state variable $X=(X_t)_{t\geq 0}$ describes a stochastic factor impacting the productivity of the firm (e.g. temperature, price of an input to production). The size of the firm $\theta$ is assumed to not impact the state parameter $X$ and the running costs $-p(\theta)q(t,x)$ describe the rate at which the firm generates profit while at the size $\theta$ given the random state of the world $x$. The assumption that $p$ is increasing and concave indicates that the larger the firm the more profit is being made but that the rate at which the firm benefits from extra units of capacity is decreasing. 
 
A prime example of a class of functions satisfying these assumptions are `Cobb-Douglas' functions where for some constants $C,\alpha,\beta \in \mathbb{R}$
\begin{equation*}
p(\theta)= C\theta^{\alpha} \quad ; \quad q(t,x)=e^{x\beta}.
\end{equation*}
The singular control problem (\ref{irreverible investment problem}) is studied using this type of production function in \cite{Bertola}. This problem has been studied via a  free-boundary problem in \cite{Kobila} using a general profit function which need not have the characteristics described here. More recently, the problem has also been tackled by applying first order conditions similar to those derived in Theorem \ref{Theorem Maximum principle} when $q(x)=e^{x}$ and $X$ is a L\'{e}vy process in \cite{RiedelSu} and \cite{Bank}. This section adds to this literature by using Theorem \ref{Theorem connection} to link the control problem (\ref{irreverible investment problem}) to an optimal stopping problem for which it is possible to derive a semi-explicit solution using the results in Section \ref{Section L_stopping}. The main result in this section is the next theorem which provides a semi-explicit characterisation of the singular control attaining the supremum in (\ref{irreverible investment problem}).

\begin{theorem}
\label{theorem multiplicative} Suppose that the intervention costs satisfy $k_{t}=K>0$ for all $t\geq 0$, $q(t,x)=q(x)$ is strictly increasing and $p_{\theta}(\theta):=\frac{d}{d\theta}p(\theta)$ is strictly decreasing. The inverse of $p_{\theta}$ is denoted $p_{\theta}^{-1}$. Let
\begin{equation}
L(x)=p_{\theta}^{-1}\left( \frac{rK}{E_0\left[q(x+\underline{X}_{T(r)})\right]} \right) \label{L multiplicative}
\end{equation}
and suppose that the functions $p,q$ are such that $L$ is well-defined for all $x\in \mathbb{R}$. Then the optimal control for (\ref{irreverible investment problem}), denoted $\theta^{\ast}$, is given by
\begin{equation}
\theta_{t}^{\ast}=\sup_{0\leq u\leq t}L(X_{u}) \label{theta m}
\end{equation}
for all $t\geq 0$. 
\end{theorem}

\begin{proof}
First assume that $p_\theta$ and $q$ are both bounded, then also the derivative of $\theta \mapsto c(t,x,\theta)=-p(\theta)q(x)$ is bounded and hence Assumption 2 is satisfied. Assumption 3 is satisfied since $c(t,x,\theta)=-p(\theta)q(x)$ is negative. Let $a$ be such that $0<q(x)<a$ for all $x \in \mathbb{R}$. Since $p_\theta$ is a positive, decreasing function with $p_\theta(\infty)=0$, $p_{\theta}^{-1}$ is decreasing with domain $(0,p_\theta(0)]$ and range $[0,\infty)$. The indirect cost function $C$ defined in (\ref{indirect cost}) satisfies
\begin{equation*}
C(x) = \inf_{z \geq 0} (-p(z)q(x)+rKz) = -p \left( z^*(x) \right) q(x) + rKz^*(x)
\end{equation*} 
where 
\begin{equation*}
z^*(x) = \left\{ \begin{array}{cl}
0 & \text{if $p_\theta(0)<rK/q(x)$}, \\
p_\theta^{-1}(rK/q(x)) & \text{if $p_\theta(0) \geq rK/q(x)$.}
\end{array} \right. .
\end{equation*}
In particular, since $q(x) \leq a$ and $p_\theta^{-1}$ is decreasing we have $z^*(x) \leq p_\theta^{-1}(rK/a)$. As $p$ is concave $p(\theta) \leq p_\theta(\theta_0) (\theta-\theta_0)+p(\theta_0)$ for any $\theta_0>0$ so
\begin{equation*}
C(x) \geq (- p_\theta(\theta_0) (z^*(x)-\theta_0) -p(\theta_0)) q(x) + rKz^*(x) \geq A - a p_\theta(\theta_0) z^*(x) 
\end{equation*}
for a constant $A$. Since $z^*(x) \leq p_\theta^{-1}(rK/a)$ this shows that $C$ is bounded from below so Assumption 4 holds.

Consequently, under the extra assumption that that $p_\theta$ and $q$ are both bounded the results in Section \ref{Section L_max} apply. If $p_\theta$ and/or $q$ is not bounded then an approximation using bounded functions can be applied analogue to the argument in Theorem \ref{theorem monotone} above.

First we show that $\theta^* \in \mathcal{A}$ i.e. that
\begin{equation}
E_x \left[ \left( \int_0^\infty e^{-rt} d\theta^*_t \right)^2 \right]<\infty. \label{control in A2} 
\end{equation}
As $q$ is bounded there exists $a$ such that $0<E_0[q(x+\underline{X}_{T(r)})] \leq a$ for all $x \in \mathbb{R}$. Hence
\begin{equation}
\frac{rK}{E_0[q(x+\underline{X}_{T(r)})]} \geq \frac{rK}{a} \label{Almost condition} 
\end{equation}
for all $x \in \mathbb{R}$. Since $p_{\theta}^{-1}$ is a decreasing function (\ref{Almost condition}) yields that $0 \leq L(x) \leq p_{\theta}^{-1}(rK/a)$. Thus $\int_0^\infty e^{-rt} d\theta^*_t \leq p_{\theta}^{-1}(rK/a)$ and in particular (\ref{control in A2}) holds.

For $l\geq 0$ define 
\begin{equation*}
\kappa(x;l)=\frac{p_{\theta}(l)}{r}E_0\left[q(x+\underline{X}'_{T(r)}) \right]-K
\end{equation*}
where $T(r)\sim \exp(r)$ and $X'$ is an independent copy of $X$ which is also independent of $T(r)$. Introduce a family of optimal stopping problems indexed by $l\geq 0$ using
\begin{equation}
Y(x;l) =\inf_{\sigma}E_x\left[ Ke^{-r\sigma}+p_{\theta}(l)\int_{0}^{\sigma}q(X_t)e^{-rt}\,dt \right] \label{parameterised stopping II}
\end{equation}
Theorem \ref{Theorem connection} shows that these stopping problems are linked to the solution to (\ref{irreverible investment problem}). In particular, the argument used in Corollary \ref{corollary stopping} shows that the stopping and continuation regions for (\ref{parameterised stopping II}) can be expressed as
\begin{equation*}
\mathcal{D}^l = \{x\in\mathbb{R}\,\vert\,\kappa(x;l) \geq 0 \} \quad ; \quad
\mathcal{C}^l = \{x\in\mathbb{R}\,\vert\,\kappa(x;l) < 0 \} .
\end{equation*}
Equivalently, these regions can be expressed as 
\begin{equation*}
\mathcal{D}^l = \{x\in\mathbb{R}\,\vert\,L(x) \leq l \} \quad ; \quad
\mathcal{C}^l = \{x\in\mathbb{R}\,\vert\,L(x) > l \} .
\end{equation*}
where $L(x)=\sup \left\{ l\in \mathbb{R}\,\vert\,Y(x;l)=K\right\}$ coincides with (\ref{L multiplicative}). The result now follows by applying Theorem \ref{Theorem connection}. 
\end{proof}
 
\medskip

\begin{remark} The above proof shows that when both $p_\theta$ and $q$ are bounded then also $L$ is bounded and in particular $\theta^* \in \mathcal{A}$. If $L$ is unbounded it does not seem so clear anymore whether $\theta^* \in \mathcal{A}$ still holds, as this not only depends on $X$ but also on $p$ and $q$. However if $\theta^* \not\in \mathcal{A}$ the above proof shows that $\theta^*$ yields a better payoff than any element in $\mathcal{A}$ and the payoff generated by $\theta^*$ can be approximated arbitrarily closely by elements in $\mathcal{A}$.
\end{remark}

The strict positivity of the costs of expansion $K$ was important in the previous theorem as otherwise it would not be possible to set $\kappa(x;l)=0$ and solve for $l$. Denote by $\theta^y$ the singular control such that $\theta_t^y =y$ for all $t\geq 0$. When there is no cost of expansion, i.e. $K=0$ we can immediately see that 
\begin{equation*}
K(\theta^y) = E\left[\int_{0}^{\infty}p(y)q(X_t)e^{-rt}\,dt\right] \leq  E\left[\int_{0}^{\infty}p(z)q(X_t)e^{-rt}\,dt\right] =K(\theta^z)
\end{equation*}
for all $z\geq y$ as it has been assumed that $q\geq 0$ and $p$ is increasing and concave. Hence it follows that $V(x)=\lim_{z\rightarrow \infty}K(\theta^z)$ so the firm would choose to immediately expand as much as possible. In this case the supremum in (\ref{irreverible investment problem}) is not attained by a singular control in the set $\mathcal{A}$. The case that $K=0$ is non-degenerate in irreversible investment problems with non-separable running profit functions, see for example \cite{Kobila} and \cite{Zariph}.

The rest of this section is dedicated to two examples. The first is the case of a Cobb-Douglas running profit function with non-constant returns to scale. The constant returns to scale case is provided in \cite{RiedelSu} using a guess-and-verify technique.

\begin{example}[Cobb-Douglas profit function]
\label{example cobb-douglas}
Consider the singular stochastic control problem
\begin{equation}
V(x)=\sup_{\theta\in\mathcal{A}}K(\theta) = \sup_{\theta \in\mathcal{A}}E_x \left[\int_{0}^{\infty}C\left(\theta_t^{(1-\alpha)}e^{\alpha X_t}\right)^{\beta}e^{-rt}\,dt -\int_{0}^{\infty}e^{-rt}\,d\theta_{t}\right] \label{irreverible example 1}
\end{equation}
where $C>0$ is a constant and the parameter $\alpha \in (0,1)$ describes the impact of each factor on the profit. The first factor in our model is the `size' of the firm $\theta$ and the second $e^{X_t}$ could be the price of an input or measure economic conditions. The constant $\beta>0$ is the `returns to scale'. When $\beta<1$ the firm is said to have `decreasing returns to scale' and when $\beta>1$ the firm is said to have `increasing returns to scale'. The case that $\beta=1$ is referred to as `constant returns to scale'. We restrict to $\beta\in (0,1/(1-\alpha))$ to ensure that Assumption 2 is satisfied by $c(t,x,\theta)=-(\theta)^{\beta(1-\alpha)}\exp(\alpha\beta x)$. Let $\gamma :=\alpha\beta/(1-\beta(1-\alpha))$, $A := C\beta(1-\alpha)$. The `indirect cost' function (\ref{indirect cost}) is 
\begin{equation*}
C(x):=\inf_{\theta\geq 0}(r\theta- C\theta^{\beta(1-\alpha)}e^{\alpha\beta x}) = 
A e^{\gamma x} - C\frac{A}{r}e^{\gamma\beta(1-\alpha) x}e^{\alpha\beta x} = A'e^{\gamma x}
\end{equation*}
where $A'$ is a constant. Suppose that the L\'{e}vy measure of $X$, denoted $\nu$, satisfies
\begin{equation*}
\int_{ x \geq 1}e^{\gamma x} \nu(dx) <+\infty .
\end{equation*}
This ensures that the Laplace exponent of $X$ is finite at $\gamma$ i.e. $\psi(\gamma):=\frac{1}{t}\log E[e^{\gamma X_{t}}]<+\infty$. Consequently, the Esscher transform
\begin{equation*}
\left. \frac{dP^{\gamma}_0}{dP_0} \right\vert_{\mathcal{F}_{t}}= e^{\gamma X_{t}-\psi(\gamma)t}
\end{equation*}
defines an equivalent probability measure. Let $E^{\gamma}_0$ denote the expectation under $P^{\gamma}$ so we may write
\begin{align*}
-K(\theta) &\geq E_x\left[\int_{0}^{\infty}C(X_t)e^{-rt}\,dt\right] = A'e^{\gamma x}E_0\left[\int_{0}^{\infty}e^{\gamma X_t-rt}\,dt\right] \\ &=A'e^{\gamma x}E^{\gamma}_0\left[\int_{0}^{\infty}e^{-(r-\psi(\gamma))t}\,dt\right]
 = Ae^{\gamma x}\int_{0}^{\infty}e^{-(r-\psi(\gamma))t}\,dt.
\end{align*}
Hence Assumption 4 is satisfied if we assume that $r-\psi(\gamma)>0$. 

Take $q(x):=\exp(\beta\alpha x)$ and $p(\theta)=C\theta^{\beta(1-\alpha)}$ so that 
$p_{\theta}(\theta)=A\theta^{\beta(1-\alpha)-1}$ and $(p_{\theta})^{-1}(y) = (y/A)^{1/(\beta(1-\alpha)-1)}$. Inserting this into (\ref{L multiplicative}) gives
\begin{equation}
L(x)=e^{\gamma x}\left(\frac{A}{r}E_0\left[e^{\beta\alpha \underline{X}'_{T(r)}} \right] \right)^{1/(1-\beta(1+\alpha))}=: \delta e^{\gamma x}. \label{example L}
\end{equation}
Consequently, applying Theorem \ref{theorem multiplicative} gives that the optimal control for (\ref{irreverible example 1}) satisfies
\begin{equation*}
\theta_{t}^{\ast} =\sup_{0\leq u\leq t} \delta e^{\gamma X_u}
\end{equation*}
In particular, when $C=1/(1-\alpha)$, $\beta=1$ we have $\gamma = 1$ and (\ref{example L}) simplifies to 
\begin{equation*}
L(x)=e^{x}\left(\frac{1}{r}E_0\left[e^{\alpha \underline{X}'_{T(r)}} \right] \right)^{1/\alpha}
\end{equation*}
which coincides with the explicit form of the solution provided in \cite{RiedelSu} Section 6 when their depreciation rate $\delta$ is taken to be zero.
\end{example}

In the previous example the optimal marginal profit at $t\geq 0$ is defined as
\begin{align*}
p_{\theta}(\theta_t^{\ast})q(t,X_t) &:= A(\theta^{\ast}_t)^{\beta(1-\alpha)-1}e^{\alpha\beta X_t}
= \frac{r}{E_0\left[e^{\beta\alpha \underline{X}_{T(r)}} \right]}\left(\sup_{0\leq u\leq t}e^{\gamma X_u}\right)^{\beta(1-\alpha)-1}e^{\alpha\beta X_t} \\ &= \frac{r}{E_0\left[e^{\beta\alpha \underline{X}_{T(r)}} \right]}\inf_{0\leq u\leq t}e^{-\alpha\beta (X_u-X_t)} \leq \frac{r}{E_0\left[e^{\beta\alpha \underline{X}_{T(r)}} \right]}
\end{align*}
where the inequality holds with equality at the points of increase of $\theta^{\ast}$. This optimal marginal profit indicates that the solution to (\ref{irreverible example 1}) is to keep $\theta$ sufficiently large as to prevent the marginal profit from production from slipping below the level $r/E_0[\exp(\alpha\beta \underline{X}_{T(r)})]$.

Our final example deals with the case that the intervention costs is a
L\'{e}vy process which could be correlated with the process $X$. At first glance this problem
appears beyond the scope of Theorem \ref{theorem multiplicative} but we use an appropriate change of measure to transform the objective functional so that Theorem \ref{theorem multiplicative} can be applied. 

\begin{example}[Independent intervention costs]
\label{example Levy costs}
Let $X,Y$ be two independent L\'{e}vy processes with $X_0=x$ and $Y_0=0$ under $P_x$ for all $x\in \mathbb{R}$. The characteristic triplet of $X$ (resp. $Y$) is denoted $(\mu_X,\sigma_X,\nu_X)$ (resp. $(\mu_Y,\sigma_Y,\nu_Y)$). Suppose that the intervention costs are of the form $k=(k_{t})_{t\geq 0}=(e^{Y_t})_{t\geq 0}$. Consider the problem (\ref{irreverible investment problem}) where $q(t,x)=e^x$  for all $t\geq 0$ and $q(\theta)=\frac{1}{1-\alpha}\theta^{1-\alpha}$ for some constant $\alpha \in(0,1)$. Under these assumptions (\ref{irreverible investment problem}) reads
\begin{equation}
V(x)=\sup_{\theta \in\mathcal{A}}K(\theta) = \sup_{\theta \in\mathcal{A}}E_x\left[ \int_{0}^{\infty}\frac{1}{1-\alpha}\theta^{1-\alpha}_t e^{X_t}e^{-rt}\,dt - \int_{0}^{\infty}e^{Y_t-rt}\,d\theta_{t}\right] . \label{another irreversible obj}
\end{equation}
We assume that the measure $\nu_Y$ satisfies
\begin{equation*}
\int_{ x \geq 0} e^{x}\,\nu_Y(dx)<+\infty . 
\end{equation*}
This assumption ensures that the Laplace exponent of $Y$, denoted $\psi_{Y}(c):=\frac{1}{t}\log E[e^{cY_t}]$ satisfies
$\psi_Y(1)<\infty$. Define an equivalent probability measure $\widetilde{P}_x$
via the density
\begin{equation*}
\left. \frac{d\widetilde{P}_x}{dP_x} \right\vert_{\mathcal{F}_{t}}=e^{Y_{t}-\psi_{Y}(1)t}.
\end{equation*}
Denote by $\widetilde{E}_x$ the expectation under the measure $\widetilde{P}_x$ and assuming that $\widetilde{r}:=r-\psi_{Y}(1)>0$ we may write the objective function in (\ref{another irreversible obj}) as
\begin{equation}
K(\theta) = \widetilde{E}_x\left[\int_{0}^{\infty}\frac{1}{1-\alpha}\theta^{1-\alpha}_{t}
 e^{X_t-Y_t}e^{-\widetilde{r}t}\,dt -\int_{0}^{\infty}e^{-\widetilde{r}t}\,d\theta_{t}\right] .
\label{step uncorrelated example}
\end{equation}
According to \cite{Kyp} Theorem 3.9, the measure $\widetilde{P}_x$ is structure preserving in the sense that $Y$ remains a L\'{e}vy process under $\widetilde{P}_x$ with characteristic triplet $(\widetilde{\mu}_Y,\sigma_Y,\widetilde{\nu}_Y)$ where $\widetilde{\nu}_Y(dx):=e^{x}\nu_Y(dx)$ and
\begin{equation*}
\widetilde{\mu}_Y:=\mu_Y - \sigma^{2} + \int_{\vert x\vert >0}x(1-e^{x})\,\nu_Y(dx) .
\end{equation*}
As we have assumed that $X$ and $Y$ are independent, this change of measure does not affect the dynamics of $X$. Let $Z:=Y-X$ then $Z$ is a L\'{e}vy process under $\widetilde{P}_x$ with triplet $(\mu_X-\widetilde{\mu}_Y,(\sigma_X^2 + \sigma_Y^2)^{1/2} ,\nu_X-\widetilde{\nu}_Y)$. It is now possible to derive a solution to the transformed problem (\ref{step uncorrelated example}) using the method used in Example \ref{example cobb-douglas}. Assume that
\begin{equation*}
\int_{\vert x\vert > 0} e^{x}\,(\nu_X(dx)-\widetilde{\nu}_Y(dx))<+\infty
\end{equation*}
and that the Laplace exponent of $Z$ satisfies $r-\psi_Z(1)>0$, then the optimal control for (\ref{another irreversible obj}) is given by
\begin{equation*}
\theta^{\ast}_{t}=\sup_{0\leq u\leq t} \beta e^{Z_{u}/\alpha}
\end{equation*}
for all $t\geq 0$ where 
\begin{equation*}
\beta^{\alpha}:=\frac{1}{\widetilde{r}}\widetilde{E}_0\left[\exp\left(\underline{Z} _{T(\widetilde{r})}\right)\right]
\end{equation*}
and $T(\widetilde{r}) \sim \exp(\widetilde{r})$ is independent from $Z$.
\end{example}

\end{document}